\documentclass{article}

\usepackage{graphicx}
\usepackage{color}	
\usepackage{inputenc}
\usepackage{amsmath}
\usepackage{amssymb}
\usepackage[nice]{nicefrac}
\usepackage{enumitem}
\usepackage{changepage}
 \usepackage{hyperref}
 \hypersetup{colorlinks = false,hidelinks}

\newtheorem{thm}{Theorem}
\newtheorem{prop}{Proposition}
\newtheorem{lem}{Lemma}
\newtheorem{rem}{Remark}
\newtheorem{ass}{Assumption}
\newtheorem{clm}{Claim}
\newtheorem{corol}{Corollary}

\newenvironment{proposition}{\begin{prop}}{\hfill \mbox{\small$\square$} \end{prop}}
\newenvironment{remark}{\begin{rem}}{\hfill $\bullet$ \end{rem}}
\newenvironment{lemma}{\begin{lem}}{\hfill  \mbox{\small$\square$} \end{lem}}
\newenvironment{corollary}{\begin{corol}}{\hfill  \mbox{\small$\square$} \end{corol}}
\newenvironment{assumption}{\begin{ass}}{\end{ass}} 
\newenvironment{claim}{\begin{clm}}{\hfill \mbox{\small$\square$}\end{clm}} 
\newenvironment{proof}[1][\!]{\noindent\hspace{1em}{\itshape Proof #1:~}}{\hfill\small$\blacksquare$}

\allowdisplaybreaks

 \def\vro{{\boldsymbol \varrho}}
 \def\vphi{{\boldsymbol \varphi}}
 \def\bGma{{\boldsymbol \Gamma}}
 \def\bUps{{\boldsymbol \Upsilon}}
 \def\bPsi{{\boldsymbol \Psi}}
 \def\btta{{\boldsymbol \theta}}
 \def\bOmga{{\boldsymbol \Omega}}
 \def\omga{{\boldsymbol \omega}}
 \def\tomga{\tilde{\boldsymbol \omega}}
 \def\bPhi{{\boldsymbol \Phi}}
 \def\dbPhi{\dot{\boldsymbol \Phi}}
 \def\bXi{{\boldsymbol \Xi}}
 \def\bxi{{\boldsymbol \xi}}
 \def\bzeta{{\boldsymbol \zeta}}
 \def\hbtta{\hat{\boldsymbol \theta}}
 \def\tbtta{\tilde{\boldsymbol \theta}}
 \def\t{{\boldsymbol \tau}}

\def\dLm{\lambda_m\{\D_L\}} \def\dLm{d_{Lm}}
\def\Pm{\lambda_m\{\P\}} \def\Pm{p_m}
\def\dM{\lambda_M\{\D\}} \def\dM{d_M}
\def\dm{\lambda_m\{\D\}} \def\dm{d_m}

\def\q{{\bf q}}
\def\dq{{\bf \dot q}}
\def\ddq{{\bf \ddot q}}
\def\tq{{\bf {\tilde q}}}

 \def\sat{{\rm sat}}
 \def\g{{\bf g}}
 \def\e{{\bf{e}}}
 \def\x{{\bf x}}
 \def\y{{\bf y}}
 \def\z{{\bf z}}
 \def\C{{\bf C}}
 \def\D{{\bf D}}
 \def\I{{\bf I}}
 \def\F{{\bf F}}
 \def\M{{\bf M}}
 \def\P{{\bf P}}
 \def\Y{{\bf Y}}

\def\q{{ q}}
\def\dq{{ \dot q}}
\def\ddq{{ \ddot q}}
\def\tq{{\tilde q}}

 \def\sat{{\rm sat}}
 \def\g{{ g}}
 \def\e{{{e}}}
 \def\x{{ x}}
 \def\y{{ y}}
 \def\C{{ C}}
 \def\D{{ D}}
 \def\I{{ I}}
 \def\F{{ F}}
 \def\M{{ M}}
 \def\P{{ P}}
 \def\Y{{ Y}}

 \def\vro{{ \varrho}}
 \def\vphi{{ \varphi}}
 \def\bGma{{ \Gamma}}
 \def\bUps{{ \Upsilon}}
 \def\bPsi{{ \Psi}}
 \def\btta{{ \theta}}
 \def\bOmga{{ \Omega}}
 \def\omga{{ \omega}}
 \def\tomga{\tilde{\omega}}
 \def\bPhi{{ \Phi}}
 \def\dbPhi{\dot\Phi}
 \def\bXi{{ \Xi}}
 \def\hbtta{\hat{ \theta}}
 \def\tbtta{\tilde{ \theta}}
 \def\t{{ \tau}}

\usepackage{comment} 
\specialcomment{keepaside}{\begingroup 
\begin{adjustwidth}{10mm}{10mm}
    \color{gray} 
                          {\tt \%-----some stuff to keep for now---${\tiny\boldsymbol{\vee}}$} \\ }
                          { \mbox{}\\ {\tt \%-----some stuff to keep for now---$^{\tiny\boldsymbol{\wedge}}$}
\end{adjustwidth}
\color{black}\normalsize\endgroup}
\definecolor{myred}{rgb}{0.8,0.1,0.16}
\definecolor{myblue}{rgb}{0,.27,.58}
 
\newif\ifitsijrnlc

\newif\ifitsdraft

\ifitsdraft
  \usepackage{refcheck}
  \def\al#1{\color{myred}{\sf  ---[AL: #1]---}\color{black}}
  \def\mal#1{\mbox{\color{myred}{\sf  ---[AL: #1]---}\color{black}}}
 \else 
   \def\al#1{}    \def\mal#1{} 
   \excludecomment{keepaside}
 \fi 

\begin{document}

\title{\LARGE On Composite Adaptive Continuous Finite-Time Control of a class of Euler-Lagrange systems}


\author{Emmanuel Cruz-Zavala \quad Jaime A. Moreno \quad Antonio Lor\'{i}a
  \thanks{E.~Cruz-Zavala is with Dept. of Computer Science, University of Guadalajara, Guadalajara,  Mexico; e-mail: emmanuel.cruz1692@academicos.udg.mx. J. Moreno is with Engineering Institute, Universidad Nacional Aut{\'o}noma de M{\'e}xico (UNAM), Mexico City, Mexico; e-mail:  JMorenoP@iingen.unam.mx. A. Lor\'ia is with Laboratoire des signaux et syst\`{e}mes (L2S), CNRS, 91192 Gif-sur-Yvette, France; e-mail:  antonio.loria@cnrs.fr 
}}

\maketitle

\begin{abstract}

 In this paper, we propose several set-point control schemes for achieving finite-time regulation in a class of Euler--Lagrange systems with $n$ degrees of freedom and uncertain potential energy. The proposed controllers are based on composite adaptive control approaches. Each control scheme consists of two main components: a Proportional--Derivative (PD)-based nonlinear feedback term and a  finite-time parameter estimation law. The estimation laws rely on the Dynamic Regressor Extension and Mixing (DREM) technique, which can be designed using either the Kreisselmeier or the least-squares dynamic regressor extensions. These results extend recent advances in finite-time adaptive control for Euler-Lagrange systems.
To the best of the authors' knowledge, the  composite adaptive control formulation proposed here, which does not employ  well-known Slotine and Li adaptive control structure, has not been studied in detail yet. The properties of the proposed controllers are rigorously analyzed using Lyapunov methods. The performance of the controllers is thoroughly evaluated through extensive simulation studies.
\end{abstract}
         
\section{Introduction}

Adaptive control is a standard approach for handling parametric uncertainty in linearly parameterizable Euler-Lagrange (EL) systems. Classical schemes often rely on passivity and virtual reference trajectories to achieve asymptotic regulation or tracking. While alternative methods, see \cite{TOM1,Kelly1997}, avoid virtual references, both approaches share a common limitation regarding parameter convergence: robust estimation requires the regressor matrix to be Persistently Excited. Although this condition can be verified for tracking tasks, it is rarely satisfied in regulation scenarios.

The Persistency of Excitation (PE) requirement can be relaxed using dynamic regressor extensions, substituting it with excitation over a finite-time interval. An interesting development in this field is the Dynamic Regressor Extension and Mixing (DREM) technique, which relax the PE condition. DREM-based estimators offer key advantages: they utilize scalar gains that allow for independent component tuning without affecting other transients; and the structure ensures parameter estimates with monotonic transient responses, eliminating overshoots and oscillations. DREM have been applied to EL system identification via power balance equations \cite{RoOrBo2021}, nonlinearly parameterized regressions \cite{OrGroNuat2021}, and has been extended to enable Finite-Time (FT) parameter estimation, \cite{WaEfBo2021}. DREM-based estimators are generally employed within an indirect approach to the design of asymptotic adaptive control schemes,  \cite{OrGroNuat2021,RoBhaKa2020}.

Composite adaptive control is incorporating DREM-based estimators to enhance convergence in nonlinear systems. While recent works (e.g., \cite{Arteaga2021}) have successfully combined gradient algorithms with DREM to achieve FT parameter estimation, the resulting trajectory-tracking typically remains asymptotic. The recent research \cite{RomOr2024} refines these architectures by incorporating mixing steps into composite frameworks, it resembles the classical composite adaptive control schemes in \cite{SlotLee1989} to improve overall stability and performance. Although indirect adaptive schemes for FT regulation using DREM have studied in \cite{CuzMoToA2025}, a composite version has not yet been developed. 

Recently, FT control has gained significant attention due to its better robustness, faster transient response, and higher accuracy than asymptotic methods, \cite{Venkataraman1993, Bhat2000}. On the one hand, research into FT adaptive trajectory-tracking has been addressed from two main approaches. The first one focuses on gain adaptation for uncertainty upper bounds and makes use of Terminal Sliding Mode (TSM) methods \cite{ManYu1999, VaMaGe2019}. While early TSM designs for robotics often encountered singularities \cite{ManYu1999}, modern TSM control have addressed this issue by applying nonsingular sliding surfaces, \cite{VaMaGe2019}. The second approach focuses on the use of nonlinear virtual references \cite{Tang1998} and the structural properties of the dynamical system, \cite{NaMaHeReBa2015,YaHuLiGu2017}. For instance, in \cite{NaMaHeReBa2015}, a  composite scheme combining TSMs, dynamic regressor extensions \cite{NaMaHeReBa2015} and discontinuous estimators is proposed. However, the TSM designs employ discontinuous terms that induces chattering effect, which is not suitable for practical implementation. The composite scheme of \cite{YaHuLiGu2017} employs integral TSMs along H\"{o}lder-continuous terms to smooth control signals but introduce high-degree terms that cause actuator saturation. On the one hand, other approach to adaptive FT control was developed for systems in strict-feedback form,  \cite{HongWaChe2006}, and in general, requires homogeneity properties, \cite{ZiEfPo2019}. It is worth emphasizing that these controllers guarantee FT stabilization without requiring PE conditions or accurate parameter estimation, \cite{YaShaChen2019}. Although these methods are effective as stabilizers, they face significant challenges in practical applications. They cannot be directly applied to FT regulation toward arbitrary equilibria, often resulting in only practical stability (steady-state error) when applied to EL systems or robot manipulators, \cite{CaiXiGuo2016}. Due to this limitation, existing solutions often rely on restrictive assumptions-such as diagonal inertia matrices or the absence of gravity \cite{Huang2014}.  This excludes most real-world mechanical systems. Consequently, a gap remains in achieving rigorous FT regulation for general EL systems using these techniques.

Most existing schemes rely on virtual reference variables, neglecting simpler alternative techniques (like Tomei's approach) that could reduce controller complexity. None of the cited works explicitly define the specific conditions required to transition from tracking to successful FT regulation. In this direction, this paper extends previous results \cite{CuzMoToA2025} by proposing a composite set-point adaptive controller that ensures both FT regulation and FT parameter estimation for fully actuated EL systems. The scheme effectively handles uncertainties in both the inertia matrix and potential forces. The design 
utilizes a least-squares algorithm for FT estimation, providing greater design flexibility and improved performance. Nevertheless, the Kreisselmeier's dynamic regressor extension can also be used with the proposed adaptive control scheme. A key distinction from existing literature is the elimination of the virtual reference variable in the control design. This significantly simplifies the practical implementation of the controller.  Beyond the controller itself, the main contribution of this work is establishing (uniform) global FT stability via Lyapunov's direct method.

The remainder of the paper is organized as follows. Section \ref{SecProblem} introduces the model and the problem statement. In Section \ref{SecMainR}, we present 
the composite adaptive finite-time control scheme. Using Lyapunov methods, we characterize the stability properties of an EL system in feedback with the proposed adaptive control scheme in Section \ref{sec:main}, this constitutes our main results. For the sake of clarity and simplicity of presentation, the detailed proofs are deferred to Appendix \ref{SecProofdetail}. Some simulation studies that illustrate the performance of the proposed adaptive controllers are given in Section \ref{SecSimExa}. Finally, brief conclusions are drawn in Section \ref{SecConcl}.

%


{\bf Notation.} 
We define $\lceil z\rfloor^q:=|z|^q\mbox{sign}(z)$, for any $z\in\mathbb{R}$ and any $q>0$, where $\mbox{sign}(z):= -1$ if $z<0$, $\mbox{sign}(z):= 0$ if $z=0$, and $\mbox{sign}(z):= 1$ if $z>0$. Correspondingly, for  $\mathbf{z}\in \mathbb{R}^{m}$, we define 
$
  \lceil\mathbf{z} \rfloor^{q}:=\begin{bmatrix}
	   \lceil z_{1} \rfloor^{q}  \  \lceil z_{2} \rfloor^{q} \ \cdots \ \lceil z_{n} \rfloor^{q} 
\end{bmatrix}^{\top}$,  $ |\mathbf{z}|^{q}:=\begin{bmatrix}
                                  	  |z_{1}|^{q}   \    |z_{2}|^{q}   \  \cdots  \  | z_{n}|^{q} 
                                  \end{bmatrix}^{\top}$ and $\mbox{sign}({\z}):=\begin{bmatrix}
	\mbox{sign}(z_{1}) \ \mbox{sign}(z_{2})   \  \cdots   \  \mbox{sign}(z_{n}) 
\end{bmatrix}^{\top}$.

\section{Model and problem Statement} \label{SecProblem}

Consider the model fully actuated $n$-DOF Euler-Lagrange systems,
\begin{equation}\label{237}   
\M(\q)\ddq + \C(\q,\dq)\dq + \g(\q) = \t,
\end{equation}
where $\q,\dq, \ddq \in \mathbb{R}^{n}$ are the vectors of joint positions, {velocities and accelerations} respectively, ${\M(\q)}\in \mathbb{R}^{n\times n}$ denotes the inertia matrix, $\C(\q,\dq)$ is the matrix of Coriolis and centrifugal forces, which is defined via the Christoffel symbols of the first kind \cite{SPOVID}, hence 
\begin{equation}
  \label{205} \dot \M(\q)=\C^{\top}(\q,\dq)+\C(\q,\dq). 
\end{equation}
Furthermore,  $\g(\q)$ is the vector of potential-energy forces, which is defined as  
\begin{equation}
  \label{209} \g(\q) := \nabla_{\q}\!\mathcal U(\q):= \frac{\partial \mathcal U}{\partial \q}(\q) 
\end{equation}
and $\mathcal U(\q)$ represents the system's potential energy. Finally,  $\t \in \mathbb R^n$ is the control input. Also, we impose the following
\begin{assumption}\label{ass1}
Let the inertia matrix and the potential energy such that
\begin{equation}\label{250}
	\M(\q) = \sum_{k=1}^{\imath}\M_k(\q){}\theta_{Mk},\quad 
        \mathcal U(\q) = \sum_{k=1}^{\jmath}{\mathcal U}_k(\q){}\theta_{Uk},
\end{equation}
where $\btta_{M}= [\theta_{M1}\ \cdots\ \theta_{M \imath}]^{\top}$ and $\btta_{U} = [\theta_{U1} \ \cdots\ \theta_{U\jmath}]^{\top}$ are constant unknown parameters. In addition, we assume that there exist known constants $\bar \theta_{j}>0$, for all $j\le \jmath$, such that $\btta_{U}\in \Theta_{U}:=\{\btta_{U}\in\mathbb{R}^{\jmath}: -\bar \theta_{j}\le \theta_{Uj}\le \bar \theta_{j}\}$, {\it idem} for $\Theta_M$;  that  $\M_k$ and ${\mathcal U}_k$ are known smooth functions, uniformly bounded; and that $\nabla_{\q}\!\mathcal U_j$ is also uniformly bounded and locally Lipschitz.  
\end{assumption}
Assumption \ref{ass1} is little conservative since $\M_i$ and $\mathcal U_i$ are often defined by constants and trigonometric functions. For instance, this is the case for robot manipulators that have only prismatic or only revolute joints. The following properties are consequence of Assumption \ref{ass1} and are fundamental for the control design. 

\begin{itemize}
%
\item[(P1)]\label{wasA2} 
There exist constants $\mu_m$ and $\mu_M>0$, such that $\mu_m\I\leq \M(\q)\leq \mu_M\I$ for all $\q\in \mathbb{R}^n$;
\item[(P2)] $(\q,\dq)\mapsto \C(\q,\dq)$ is globally Lipschitz in $\dq$, uniformly in $\q$ (with Lipschitz constant $L_c > 0 $).
\item[(P3)] From \eqref{209} and \eqref{250}, we have 
  \begin{equation}\label{415} 
    \g(\q) = \sum_{k=1}^{\jmath}\nabla_{\q}\!{\mathcal U}_k(\q)\theta_{Uk} =: \bPsi(\q)\btta_{U},
  \end{equation}
where $\btta_{U}\in \Theta_U$ is defined above and  $\bPsi(\q)\in\mathbb{R}^{n\times \jmath}$ is the so-called regressor matrix, which under Assumption \ref{ass1}, is uniformly bounded and globally Lipschitz. Therefore, there exists $L_{g}>0$ such that $\|\g(\x)-\g(\y)\|\leq L_{g}\|{\x-\y}\|$ for all $\x$, $\y \in \mathbb R^n$. Furthermore, for each $i\leq n$ there exists $k_{gi}\geq 0$ such that $\sup_{\q\in\mathbb{R}^{n}}\|g_{i}(\q)\|\leq k_{gi}$.
\end{itemize}

For EL systems \eqref{237} satisfying Assumption \ref{ass1}, the control goal is to achieve finite-time stabilization of the equilibrium $\{(\q,\dq)=(\q_d,0)\}$, for any given $\q_d\in \mathbb R^n$. More precisely, the origin of the closed-loop system is required to be uniformly finite-time stable \cite{Moulay2008}. Our solution consists in a simple FT-Proportional-Derivative (FT-PD) controller---{\it cf. } \cite{CuNuMo2020}, which is defined by 
\begin{equation}\label{278}
 	\tau= -\P\lceil \q-\q_d\rfloor^{a}-\D\lceil \dq\rfloor^{b}-\D_L\dq+\bPsi(\q){\hbtta}_{U},
 \end{equation} 
where $\P>0$ and $\D,\D_L>0$ are the {\it proportional} and {\it derivative} gains, $\hat {\btta}_{U}$ are estimates of the unknown parameters ${\btta}_{U}$ defined in Assumption \ref{ass1}. 
In addition, the exponents $a$ and  $b$ are defined as
 \begin{equation}\label{ExponeHom} 
 	a:={m_c\over r_{1}}, \quad b:={m_c\over r_{2}},\quad m_c:=2r_{2}-r_{1},\quad 2r_{2}>r_{1}> r_{2}>0.
 \end{equation}
 \begin{remark}
Note 	that $1> b > a > 0$.
 \end{remark}

Control laws of the form  \eqref{278} have been reported in the literature before--- see {\it e.g.,} \cite{CuzMoToA2025}. The main contribution in this paper resides in the update law for $\hat {\btta}_{U}$ and, more significantly in the proof of uniform global FT stability via Lyapunov's direct method. 

The interest of the FT-PD controller is its remarkable simplicity; clearly, for $a = b =1$, we recover the {\it simple PD} adaptive controller of P. Tomei \cite{TOM1}. We also highlight that a weaker notion of FT set-point stabilization for EL systems --where uniform settling time with respect to the initial conditions is not guaranteed-- was established in \cite{CuzMoToA2025}, using a controller of the form \eqref{278}, but with $\D_L = 0$ and relying on a indirect adaptive approach. Beyond this seemingly innocuous modification, the key difference lies in the fact that, in this paper, we employ a completely different estimation law to update $\hat{\btta}_U$, together with a composite adaptive control scheme. These design modifications enable us to establish uniformity (with respecto to time) of the settling time.

\section{Composite adaptive finite-time control}
\label{SecMainR}

In addition to the control law \eqref{278}, we introduce the following composite adaptation law to update the estimates $\hat {\btta}_U$:  
\begin{equation}
\dot{\hbtta}_{U}=-\bGma\bPsi\left(\q\right)^{\top}\left[\gamma_{1} d_{1}\tanh(\e_{1}) + [\gamma_1+\gamma_2]\e_{2}\right]-[\gamma_1+\gamma_2]\bGma \bUps_{U} f(\Delta){\bXi}_{\theta_{U}}, 
\label{318}
\end{equation}
where $\e_1 := \q - \q_d$ and  $\e_2 := \dq$; and $\gamma_1$, $\gamma_2$, $d_1$, $\bGma$ and $\bUps_U$ are design parameters to be specified. In particular, $\gamma_1$, $\gamma_2$, and $d_1$  as positive constants, while $\bGma$ and $\bUps_U$ as diagonal positive definite matrices of dimension $\jmath \times \jmath$. The symbols $f$, $\Delta$ and $\bXi$ are explained farther below. 

\begin{remark}\label{remBoundonTta} 
Following Assumption \ref{ass1}, we set the initial conditions of $\hbtta_{U}$ so that $\|\tbtta_U(0)\|\leq2\|\bar {\btta}\|$, where the vector $\bar {\btta}$ is given by $\bar {\btta}= \begin{bmatrix}
 \bar \theta_{1}, \ \dots, \ \bar \theta_{\jmath}	\end{bmatrix}^{\top}$ and $ \tbtta_U(0)$ denotes the initial parametric error.	
\end{remark}

The estimation law in \eqref{318} is called {\it composite} because the first term on the right-hand side is reminiscent of a classical direct adaptive control law, whereas the second term is reminiscent of an indirect adaptive control law---{\it cf.} \cite{SLOTLIB}. Indeed, akin to the latter reference, and roughly speaking, $\Delta$ can be interpreted as the determinant of a matrix generated by a dynamic regression algorithm and that is evaluated along system trajectories. The function $f: \mathbb R\to \mathbb R$ corresponds to a continuous saturation function and is defined as 
 \begin{equation}\label{fphi}
	f(\Delta):={\lceil  \Delta  \rfloor^{d}\over 1+|\Delta|^{c+d}},
\end{equation}
with constants $c$ and $d>0$, such that $c<1$. This function is introduced for technical reasons; as shown below, it plays a key role in establishing uniformity of the settling time for the closed-loop solutions. Finally, ${\bXi}_{\theta_{U}}$ is related to the so-called prediction error in the sense of Slotine \cite{SLOTLIB}; it is defined as 
\begin{equation}
\label{329}
 \bXi_{\theta_{U}} := 
\begin{bmatrix}
\lceil  \Delta {\hat \theta}_{U1} - Y_{U1} \rfloor^{c}\\
\vdots\\
\lceil \Delta {\hat \theta}_{U\jmath}-\
	Y_{U\jmath} \rfloor^{c}
\end{bmatrix},
\end{equation}
where $Y_{Uk}:\mathbb R_{\geq 0}\to \mathbb R$ is a scalar function of time, for any $k \leq \jmath$, and denotes an element of the vector
\[
\Y_U = \big[ Y_{U1}\ \cdots\ Y_{U\jmath}  \big]^\top,
\]
which is generally computable or measurable. The vector $\Y_U$ and the function $\Delta$ may be generated in different ways, using, {\it e.g.,} a least-squares algorithm or a Kreisselmeier's Dynamic Regression Extension (DRE), and different kinds of paremeterization \cite{RoOrBo2021,CuzMoToA2025}, so they are functions of time (and initial state values). For completeness, and further development, we recall next the power-balance-based parameterization and the DRE based on least-squares with forgetting factor \cite{RoOrBo2021,OrRoAr2022}.  

 \vskip 3pt
\noindent {\bf Filtered regressor based on power-balance paramterization}---see {\it e.g.,} \cite{RoOrBo2021}: Consider the total energy of the system \eqref{237},  given by 
\[
   \mathcal E_{T} :=\sum_{k=1}^{\imath}{1\over 2}\dq^{\top}{\M}_k(\q)\dq{}\theta_{Mk}+\sum_{k=1}^{\jmath}{\mathcal U}_k(\q){}\theta_{Uk}.
\]
Defining the vector of unknown parameters
\begin{equation}\label{Ttavector}
	\btta := \begin{bmatrix}
	\btta_{M}^{\top} & \btta_{U}^{\top}
\end{bmatrix}^{\top}\in \mathbb{R}^{l}, \quad l=\imath+\jmath.
\end{equation}
and $\omga : \mathbb R^{2n} \to \mathbb R^l$, 
\begin{equation}\label{Omega}
	\omga = 
              \begin{bmatrix}
		{1\over 2}\dq^{\top} \M_1\dq &\ \cdots \
	        & \
		{1\over 2}\dq^{\top} \M_{\imath}\dq & \
		\mathcal U_{1}(\q) &
		\ \cdots \ &
		\mathcal U_{\jmath}(\q)
	     \end{bmatrix}^{\top}
\end{equation}
we see that $\mathcal E_{T} = \omga^\top \btta$. Then, given a control input $\t$ (which is evidently available), on one hand, we define $\y_{\mathcal E}$ as the output of a stable filter, {\it i.e.,}
\begin{equation}\label{372}
	\dot \y = -\lambda_{1} \y+\lambda_{0}\dq^{\top}\t,       
        \quad \lambda_0, \ \lambda_1>0.
\end{equation}
On the other, using $\tomga:= \omga(\q(t,\q_\circ,\dq_\circ),\dq(t,\q_\circ,\dq_\circ))$, we define 
\begin{subequations}
  \label{361}
  \begin{eqnarray}
     \label{361a} 	\dot \z &=&-\lambda_{1}[\z+ \lambda_0\tomga],\\
     \label{361b}	\bOmga^{\top}&=&\z+ \lambda_0\tomga.
  \end{eqnarray}
\end{subequations}
Note that the latter equations correspond to the dirty derivative of $\dot{\tomga}$, that is, 
\begin{equation}
  \label{381} 
\dot{\bOmga}^{\top} = - \lambda_1 \bOmga^{\top} + \lambda_0\dot{\tomga}. 
\end{equation}
Next, we transpose all the terms on both sides of the latter, we post-multiply by $\btta$, and we use the identity $\dot{\mathcal E}_{T} = \dot\omga^\top \btta = \dot \q^\top\t$, which holds under \eqref{205}, we obtain 
\[
\dot{\bOmga}\btta = - \lambda_1 \bOmga\btta + \lambda_0 \dot \q^\top\t.
\]
From the latter and \eqref{372}, we obtain the linear regression equation 
\begin{equation}\label{LREq}
	\y=\bOmga\btta.
\end{equation}
In other words, \eqref{LREq} holds for $\y$ computed as in \eqref{372} and $\bOmga$ computed as in \eqref{361}. 
\begin{remark}
Alternatively to the previous parameterization, we can use the so-called force-balance parameterization \cite{CuzMoToA2025}, which is defined as follows. Consider the EL equations \eqref{237} rewritten in the form 
\begin{equation}\label{401}
\sum_{i=1}^{\imath}\left[ \frac{\rm d}{\rm dt} \left\{\M_k(\q)\dq\right\}-{1\over 2}\nabla_{\q}\!\left\{\dq^{\top}\M_k(\q)\right\}\right]\theta_{Mk}+ \sum_{k=1}^{\jmath}\nabla_{\q}\!{\mathcal U}_k(\q)\theta_{Uk}
=\t,
\end{equation}
and let 
\begin{equation}
 \label{399} \dot \y = -\lambda_1 \y + \lambda_{0}\t, \quad \lambda_{0},\lambda_1 > 0. 
\end{equation}
Then, we define $ \vphi_{1} = \big[ \varphi_{11} \ \cdots \ \varphi_{1\imath}  \big]$, where
\[
\varphi_{1k} := \Big[ \lambda_{0}\M_k(\q)\dq+ 
{\lambda_{0}\over 2\lambda_{1}}\nabla_{\q}\!\{ \dq^{\top}\M_k(\q)\dq \}\Big],
\]
$\vphi_{2} = \Big[\nabla_{\q}\!{\mathcal U}_{1}(\q)  \ \cdots \ \nabla_{\q}\!{\mathcal U}_{\jmath}(\q)  \Big]$, 
$ \vphi_{3} = \Big[ \M_1(\q)\dq \ \cdots \ \M_{\imath}(\q)\dq\Big]$; 
and 
\begin{subequations}\label{419} 
\begin{align}
\bOmga & := \big[ \bOmga_{\mathcal D1} \ \  \bOmga_{\mathcal D2}  \big]\\
  \dot \bOmga_{\mathcal D2} &= -\lambda_1 \bOmga_{f2} +\lambda_{0}\vphi_{2},  \\
 \bOmga_{\mathcal D1} &= \z + \lambda_{0}\vphi_{3},\\ 
 \dot \z &= -\lambda_1 (\z+ \vphi_{1}).
\end{align}
\end{subequations} 
It follows (see \cite{CuzMoToA2025}) that $\y$ as defined in \eqref{399}, with $\bOmga$ as defined in \eqref{419}, satisfies
\begin{equation}\label{NLRPE}
\y = \bOmga \btta.
\end{equation}
\end{remark}

\begin{remark}\label{rmk:t0}
Either parameterization, as defined  in \eqref{NLRPE} or \eqref{LREq} is linear; in both cases, $\y$ can be computed without knowledge of $\btta$ (see \eqref{372} and \eqref{399}). When used in indirect adaptive control, this typically requires PE of $\bOmga$. The advantage of the DRE is the ability to generate linear regression equations of the form of \eqref{372} or \eqref{399}  in which, after a mixing step, the factor of the unknown parameters results to be scalar; this leads to relaxed conditions of persistency of excitation. That said, in both cases, the respective inputs to the filters, $\tomga$ or $\vphi_j$ with $j\in \{1,2,3\}$ are, strictly speaking, functions of time through the system's trajectories. Therefore, even if we drop the arguments to avoid a cumbersome notation, it is important to  stress that $\bOmga$ and $y$ depend on the initial state values. In that regard, PE conditions (of any kind) must be imposed to hold uniformly in compacts of the state---this is particularly important in tracking control problems, in which case they also depend on the initial time \cite{al:LTVsystSCL} and yet, it has been abundantly overlooked in the literature of adaptive control systems. 
\end{remark}

Next, we leverage either of the previous parameterizations to generate a new one that generates a computable output $\Y_U := \Delta \btta$ to be used in \eqref{329} without the knowledge of $\btta$. To that end, different algorithms may be used, but  we favor the least-squares with forgetting factor DRE, including a mixing step, as proposed in	\cite{OrRoAr2022}.

\vskip 3pt
\noindent {\bf Least-squares DRE with forgetting factor plus mixing step:} Consider $\y$ and $\bOmga$ alternatively defined as in \eqref{372}-\eqref{361} {\it or} as in \eqref{399}-\eqref{419}. Let 
\begin{subequations}
\label{443}
\begin{eqnarray}
  \dot{\hat{\vro}} &=& \alpha \F \bOmga(\y - \bOmga \hat{\vro}), \qquad \hat{\vro}(0) =: \vro_0 \in \mathbb{R}^{l}, \\
    \dot{\F} & = & -\alpha \F \bOmga^{\top} \bOmga\F + \beta \F, \quad \F(0) = \frac{1}{f_0}\I_{l}, \\
    \beta & := & \beta_0 \left[1 - \frac{\|\F\|}{\xi}\right], \quad \alpha_{0},f_0,\beta_{0}>0, \ \xi \geq \frac{1}{f_0}, 
\\    \dot{z} &= & -\beta z, \quad z(0) = 1, 
\end{eqnarray}
\end{subequations}
 Then, defining further 
\begin{subequations}
  \label{488}
  \begin{eqnarray}
     \label{488a} 
           \Y &:=& \text{adj}\{\bPhi\}[\hat{\vro} - z f_0 \F \vro_0],\\
     \label{488b}
           \Delta &:=& \text{det}\{\bPhi\}, \\
     \label{488c}
           \bPhi &:=& \I_{l} - z f_0 \F.
  \end{eqnarray}
\end{subequations}
we obtain the new regression with scalar factor $\Delta$, 
\begin{equation}\label{LSREq}
 \Y = \Delta \btta. 
\end{equation}
Then, the vector $\Y_U$ needed to implement the adaptation law in \eqref{318}-\eqref{329} is given by the last $\jmath$ components of the vector $\Y$ as defined by \eqref{488a}-\eqref{443}, which, in view of \eqref{LSREq}, yields 
\begin{equation}\label{465}
\Y_U = \Delta \btta_U 
\end{equation}
Correspondingly, defining the estimate of $\Y_U$ as $\hat{\Y}_U = \Delta \hat{\btta}_U$, we see that $\bXi_{\theta_{U}}$, defined as in \eqref{329}, which takes the form 
 \begin{equation}
\label{471}
  \bXi_{\theta_{U}}:=
  \begin{bmatrix}
    \lceil  \Delta {\hat \theta}_{U,1}-Y_{(\imath+1)} \rfloor^{b}\\
    \vdots\\
    \lceil \Delta {\hat \theta}_{U,\jmath}-\
    Y_{{i},l} \rfloor^{b}
  \end{bmatrix},
\end{equation}
satisfies
\begin{equation}
  \label{357} 
\bXi_{\theta_U} = \lceil\Delta\rfloor^{b} \lceil\tbtta_U\rfloor^{b}.
\end{equation}
This is particularly convenient for the purpose of analysis, which is presented in Section \ref{sec:main}, farther below. 
\begin{rem} \label{remCramer} To avoid computing of the adjugate matrix $\mbox{adj} \{ \bPhi_{L}\}$ in numerical implementations, the elements $\Y_j$, $j\in \{1,2,\ldots,l\}$,  in (\ref{LSREq}) are computed using Cramer's rule as  $\Y_j = $~det$\{\bPhi_{\Y\!,\,j}\}$, for all $t>0$, where ${\bPhi}_{\Y\!,\,j}$ denotes the matrix $\bPhi$ with its $j$-th column replaced by the vector $\Y$---see \cite{KoAraUsh2020}.
\end{rem}


\begin{remark}
Alternatively to the least-squares DRE regression, one can also use the so-called Kreisselmeier's dynamic regressor extension. Consider $\y$ and $\bOmga$ defined as in \eqref{372}-\eqref{361} {\it or} alternatively as in \eqref{399}-\eqref{419}. Let 
\begin{subequations}
  \label{510}
  \begin{eqnarray}
     \label{510a}   
           \dbPhi_1 &=& -\lambda_2 \bPhi_1 + \lambda_3\bOmga^{\top}\y, \\ 
     \label{510b}
           \dbPhi_2 &=& -\lambda_2 \bPhi_2 + \lambda_3\bOmga^{\top}\bOmga,
  \end{eqnarray}
\end{subequations}
with initial conditions $\Phi_1(0) = 0$ and $\Phi_2(0)=0$. 
 Then, we define 
\begin{equation}\label{832} 
\Y := \mbox{adj}\{\bPhi_2\}\bPhi_1, \qquad \Delta := \mbox{det}\,\{\bPhi_2\}. 
\end{equation} 
We stress that for $\lambda_{3} = 1$, the equations reduce to Kreisselmeier's dynamic regressor extension \cite{Kreissr1977}.
\end{remark}
\color{black}

\section{Stability analysis}
\label{sec:main}

After the previous derivations, the system \eqref{237} in closed loop with the control law \eqref{278} and the adaptation law \eqref{318} becomes
\begin{subequations}\label{476} 
\begin{align}
\label{476a}
\dot{\e}_{1} & =\e_{2}, \\
\label{476b}
\dot{\e}_{2} & =\M^{-1}(\q)\left[-\P\left\lceil \e_{1}\right\rfloor ^{a}-\D\left\lceil \e_{2}\right\rfloor ^{b}-\D_L\e_{2}-\C(\q,\e_{2})\e_{2}+\bPsi(\e_1+\q_d)\tbtta_U\right]\\
\label{476c}
\dot{\tbtta}_{U} & =-\bGma\bPsi(\e_1+\q_d)^\top\left[\gamma_{1}d_{1}\tanh(\e_{1})+[\gamma_1+\gamma_2]\e_{2}\right]-[\gamma_1+\gamma_2]\bGma\bUps_{U}\zeta_{1}\left(\Delta\right)\left\lceil \tbtta_U\right\rfloor ^{b},
\end{align}
\end{subequations}
where we used \eqref{357}, and we defined $\zeta_{1}(\Delta):= f(\Delta)\lceil\Delta\rfloor^{b}$, with $f$ as in \eqref{fphi}, with $c=b$. 

Strictly speaking, the complete closed-loop system corresponds to equations \eqref{476} together with \eqref{443}-\eqref{488} and with either \eqref{399}-\eqref{419} or \eqref{372}-\eqref{361}. Note that in either case, since $\q_d$ is constant, the overall closed-loop system is {\it autonomous}. However, for the purpose of analysis, we focus only on the system described by Eqs. \eqref{476} by considering it nonlinear time-varying. The system \eqref{476} is non autonomous because $\Delta$ is a function of time. In fact, $\Delta$ is generated dynamically via the filters defined in \eqref{443}, whose inputs $\y$ and $\bOmga$ are in turn generated by either the filters in \eqref{399}-\eqref{419} or in \eqref{372}-\eqref{361}, driven by either $\vphi$ or $\omga$, which are functions of the plants states $\q = \e_1 + \q_d$ and $\dq = \e_2$. That is, $\Delta$ depends on time because it is generated by the controller's and plant's trajectories, generated in turn by the initial conditions $\bxi_\circ := [\x_\circ^\top\ {\bzeta}_\circ^\top]^\top$, where  $\x_\circ := [\e_{1\circ}^\top\ \e_{2\circ}^\top\ \tbtta_\circ^\top]^\top$ corresponding to the initial conditions of the closed-loop system \eqref{476} and ${\bzeta}_\circ$  corresponds to a vector of appropriate dimension containing the controller's initial conditions. However, the complete closed-loop system being autonomous, $\Delta$ does not depend on the initial time. In that light, without loss of generality, in the analysis of \eqref{476} we assume that $t_\circ = 0$. We stress however, that this is possible only because in this paper $\q_d$ is constant; the same cannot be considered, in general, in scenario of tracking control or in which the system is affected by exogenous functions of time. Finally, we stress that by design, $\Delta(\,\cdot\,,\xi_\circ)$ is absolutely continuous and $\Delta(t,\,\cdot\,)$ is continuous. 

The fact that $\Delta(t,\bxi_\circ)$ is independent of the initial time is fundamental because one can safely rely on conditions of persistency of excitation, without the necessary precautions to be taken for regressors that depend on (filtered) systems trajectories---see \cite{al:LTVsystSCL} for a discussion. In particular, an important property of $\Delta$ is that if $\bOmga$ is exciting in an interval, uniformly on compacts of the state, for any $R>0$, there exist constants $\Delta_m$ and $T_{s}>0$ such that, 
\begin{equation}\label{eq8}
\Delta(t,\bxi_\circ) \geq \Delta_m, \quad \forall t \geq T_{s}.
\end{equation}
We recall that, according to \cite{KreiRie1990},  we say that a bounded function $\bOmga : \mathbb R_{\geq 0}\to \mathbb{R}^{N\times p}$ is exciting in the interval $[t_1, t_1+ T_{\mbox{\footnotesize \sc ie}} ]$, with $t_1$, $T_{\mbox{\footnotesize \sc ie}} > 0$ and $T_{\mbox{\footnotesize \sc ie}}$, uniformly on compact sets of the state if, for any $R>0$, there exists $\mu >0$ such that 
\begin{equation}
  \label{502} \int_{t_1}^{t_1+ T_{\mbox{\footnotesize \sc ie}}} \bOmga(s,\bxi_\circ)^{\top}\bOmga(s,\bxi_\circ)ds \geq \mu {\I}
\end{equation}
for all $\|\bxi_\circ\|\leq R$. 
In what follows, we present our main statements on stability of the origin for \eqref{476}, under the following standing hypothesis. 

\begin{assumption}\label{ass:gains}
For the system \eqref{476} it is assumed that $\P,\D,\D_L\in\mathbb{R}^{n\times n}$, $\bGma,\bUps_{U}\in\mathbb{R}^{\jmath\times \jmath}$ are diagonal, positive-definite;  $a$ and $b$ satisfy (\ref{ExponeHom}); and  $d_{1}$ as well as  the quotient  $\gamma_{1}/(\gamma_{1}+\gamma_{2})$, are sufficiently small. On the other hand, $\q_{d}\in\mathbb{R}^{n}$ is arbitrary.
\end{assumption}
\begin{remark}
  The assumption that $\q_{d}\in\mathbb{R}^{n}$ is arbitrary is not innocuous. Indeed, to the best of the authors' knowledge, only in \cite{CuzMoToA2025} a similar assumption is considered. In contrast, in the literature on FT set-point stabilization with unknown gravitational forces, it is most often assumed that $\q_d$ corresponds to the system's natural equilibrium. 
\end{remark}
\begin{remark} For sake of simplicity in presenting the upcoming results, we define $p_{m} :=\min\limits_{1 \leq i\leq n}\{P_{i}\}$, $p_{M} :=\max\limits_{1 \leq i\leq n}\{P_i\}$, 
$d_{m}  :=\min\limits_{1 \leq i\leq n}\{D_i\}$, $d_{M}  :=\max\limits_{1 \leq i\leq n}\{D_i\}$, $d_{Lm}  :=\min\limits_{1 \leq i\leq n}\{D_i\}$, $d_{LM} :=\max\limits_{1 \leq i\leq n}\{D_i\}$, and $\lambda_m\{\cdot\}$ denotes the smallest eigenvalue of a symmetric matrix.
\end{remark}
\begin{proposition}[UGS]\label{prop:ugs}
Consider the  system \eqref{476} under Assumption \ref{ass:gains} and where $\Delta$ is any continuous function. Then, the origin $\{\x =0\}$ for the  where $\x := [\e_1^\top\ \e_2^\top\ \tbtta^\top]^\top$, is uniformly globally stable  if, in addition, 
 \begin{align}
\label{530}	
d_1\leq \dLm\mu_m/\mu_M^2.
\end{align}
\end{proposition}
\begin{proof}
Let Assumption \ref{ass:gains} generate $r_1$, $r_2$, $\Gamma$, $\gamma_1$ and $\gamma_2>0$, and let $\gamma_{A1}:=\gamma_{1}/(\gamma_{1}+\gamma_{2})$.  Consider the Lyapunov function candidate
\begin{eqnarray}
\label{SLFV1}
V_1(\e,\tbtta_U) & :=& (\gamma_1+\gamma_2)V_0(\e) +\gamma_1d_{1}[\tanh(\e_1)]^{\top}\M(\q)\e_2\\
\nonumber 
&& +\gamma_1d_1\sum_{i=1}^{n}D_{Li}\ln\big(\cosh(\tilde{e}_{1i})\big)+\frac{1}{2}\tbtta_U^{\top}\bGma^{-1}\tbtta_U\\
\label{ISSLF}
  V_0(\e) & := & {r_1\over 2r_2}\e_1^{\top}\P\lceil \e_1\rfloor^{a} + {1\over 2}\e_2^{\top}\M(\q)\e_2, 
\end{eqnarray}
where $D_{Li}$ is the $i$th element in the diagonal of $D_L$. Note that $V_1(\e,\tbtta_U) \geq   V_{nn1}(\e) +\frac{1}{2}\tbtta_U^{\top}\bGma^{-1}\tbtta_U$, where
\begin{equation}
\label{549}
  V_{nn1}\left(\e\right) :={1\over 2}\e_2^{\top}\M(\q)\e_2 +d_{1}[\tanh(\e_1)]^{\top}\M(\q)\e_2+d_{1}\sum_{i=1}^{n}D_{Li}\ln(\cosh(e_{1i})).
\end{equation}
In turn, $\|\tanh(\e_1)\| \leq \big[2\sum_{i=1}^{n}\ln(\cosh(e_{1i}))\big]^{1/2}$, this provides a quadratic lower bound of $  V_{nn1}$, from which the condition \eqref{530} assure that $  V_{nn1}(\e)$ is positive semidefinite. On the other hand, using Lemmata \ref{lemaMeanP}-\ref{lemaMeanP0} from the Appendix, we obtain after some direct but lengthy computations:
\begin{equation}
  \label{588} 
{\dot V_1(\e,\tbtta_U)\over\gamma_{1}+\gamma_{2} }\leq-{1\over 2}v_1(\e)-v_2(\e)-\lambda_m\{\bUps_{U}\}\zeta_{1}\left(\Delta\right) \|\tbtta_U\|^{1+b},  
\end{equation}
where
\begin{subequations}
 \begin{eqnarray}
  \label{585}
    v_1(\e) &:=&  \dm\|\e_2\|^{1+b} + \dLm\|\e_2\|^2 
       + {\gamma_{A1} d_{1}\Pm}\sum_{i=1}^{n}\tanh(e_{1i}) \lceil e_{1i}\rfloor^{a},
\\ 
   v_2(\e) &:= & {\dm\over2}\|\e_2\|^{1+b}
     + \gamma_{A1} {d_{1}\over 2}  \Pm\|\tanh(\e_1)\|^{1+a}
     -\gamma_{A1} d_{1}A_0n^{b-a\over 2(1+b)}\|\tanh(\e_1)\|^{1+a\over 1+b}\|\e_2\|^{b}
\nonumber\\      
  \label{587} 
     &&\ + \big[\dLm/ 2 - \gamma_{A1} d_{1}[m_2+\sqrt{n}L_c]\,\big]\|\e_2\|^2,
 \end{eqnarray}
\end{subequations}
and $A_0:=n^{(1-b)/2}\dM$. Moreover, applying Young's inequality: $x^{\mu}y^{1-\mu}\leq \mu c^{1/\mu}x+(1-\mu)c^{-{1/(1-\mu)}}y$, which  holds for any $\mu\in (0,1)$, we obtain after long computations,  that $v_2(\e)\geq 0$ if $\gamma_{A1}$ satisfies 
\begin{equation}\label{602}
\gamma_{A1} \leq   
\min\left\{
\left[{ (1-\alpha)\Pm\over 2A_0n^{b-a\over 2(1+b)}(1-\nu_2)}\right]^{ 1-\nu_2\over \nu_2}{\dm\over 2d_1A_0n^{b-a\over 2(1+b)}\nu_2}, \
\frac{\dLm}{ 2d_{1} [\mu_M+\sqrt{n}L_c]}
\right\},
\end{equation}
which is trivially satisfied by choosing $\gamma_{A1}$ sufficiently small, for any given controller and plant parameters---see Appendix \ref{app:onv2}. Thus, after \eqref{588}, we have
\begin{equation}
  \label{629} 
\dot V_1(\e,\tbtta_{U}) \leq -{\gamma_{1}+\gamma_{2}\over 2}v_1(\e)\leq 0  
\end{equation}
and the statement follows since $V_1(\e,\tbtta_{U})$ is radially unbounded. 
\end{proof}

\begin{proposition}[Asymptotic regulation]\label{prop:assconv}
Consider the system \eqref{476} under Assumption \ref{ass:gains}, let $\Delta$ be continuous and let \eqref{530} hold. Then, for any initial conditions $\x_\circ \in \mathbb R^{2n+\jmath}$, the regulation errors $\e_1$ and $\e_2$, as well as the product $\bPsi(\q)\tbtta_U$, converge asymptotically to zero.
\end{proposition}
\begin{proof}
Consider the function 
\begin{equation}
  \label{SLFVA1} 
V_{A1}(\e,\tbtta_{U}) :=  V_1(\e,\tbtta_{U}) + \big[\gamma_1^{-r_1}V_1(\e,\tbtta_{U})^{m_2}\big]^{1/2r_2}
\end{equation}
where $m_2 = 2r_2 + r_1$. Since $V_1(\e,\tbtta_{U})$ is positive definite and radially unbounded, so is $V_{A1}(\e,\tbtta_{U})$. Furthermore, after long, but straightforward computations (see Appendix \ref{app:VA1}), we obtain
\begin{align}
\dot V_{A1}(\e,\tbtta_{U}) & \leq -(\gamma_{1}+\gamma_{2})\left[\gamma_{A1}c_{1}{\bar{v}}_{1}\left(\e\right)+{\dm\over 2}\|\e_{2}\|^{1+b} \right. 
\label{TDofVA1} \\
& \qquad \qquad \left. +\eta_{2}\|\tbtta_U\|^{r_{1}\over r_{2}}\|\e_{2}\|^2+\lambda_m\{\bUps_{U}\}\zeta_{1}(\Delta)\|\tbtta_U\|^{1+b}\right],\nonumber\\
\bar{v}_{1}\left(\e\right) &= \|\e_{1}\|^{\frac{2r_{2}}{r_{1}}}+\|\e_{2}\|^2. \nonumber
\end{align}
That is, $V_{A1}(\e,\tbtta_{U})$ is a non-strict Lyapunov function for the closed-loop system \eqref{476}. Next, we construct another Lyapunov function candidate, whose derivative is negative definite in $e$ and $\z:=\bPsi(\q)\tbtta_U$. Let $m_{3}:=2r_{1}+r_{2}$ and
\begin{equation}
  \label{SLF3}
	v_{A2}(v_{A1},v_e)  := \left(\gamma_3+\gamma_{4}\right)v_{A1}
        +\gamma_{4}v_{A1}^{\frac{m_3}{2r_2}}+\gamma_3v_{e}.
\end{equation}
Then, consider the Lyapunov function candidate $V_{A2}(\e,\tbtta_U) := v_{A2}\big(V_{A1}(\e,\tbtta_U), V_{e}(\e,\tbtta_U)\big)$, where
\begin{subequations}
  \label{660}
  \begin{eqnarray}
     \label{660a} V_{e}(\e,\tbtta_U) & := & V_{A1}(\e,\tbtta_U)^{\frac{m_3}{2r_2}}-d_{2}{\tanh(\|\z\|)^{m_3-r_2\over m_c}\over 1+\|\z\|}\z^{\top}\M(\q)\e_2,\\
     \label{660b}
     d_{2}^2 &\leq&\frac{\gamma_{2}\mu_m m_{3}}{2 \Big[  k_{\Psi}^{\frac{m_{3}-r_{2}}{m_c}} \mu_M\Big]^2r_{2} }
     \left[\frac{\lambda_m\{\bGma^{-1}\}m_{3}}{2(m_{3}-r_{2})}\right]^{\frac{m_{3}-r_{2}}{r_{2}}},
  \end{eqnarray}
\end{subequations}
$d_{2}>0$, $k_{\psi} \geq \|\bPsi(\q)\|$, for all $\q\in \mathbb R^n$, and  $m_c$ is defined in (\ref{ExponeHom}). Note, from \eqref{SLFVA1}, \eqref{SLFV1}, and \eqref{ISSLF}, that $V_{A1}(\e,\tbtta_U) \geq \gamma_2\frac{\mu_m}{2}\|\e_2\|^2+\frac{1}{2}\lambda_m\{\bGma^{-1}\}\|\tbtta_U\|^2$. Therefore, 
\[
V_{e}(\e,\tbtta_U)\geq V_{nl}(\e_2,\tbtta_U):=\left(\gamma_2\frac{\mu_m}{2}\|\e_2\|^2+\frac{1}{2}\lambda_m\{\bGma^{-1}\}\|\tbtta_U\|^2\right)^{\frac{m_3}{2r_2}}-d_{2}k_{\Psi}^{\frac{m_3-r_2}{r_2}}\mu_M\|\tbtta_U\|^{\frac{m_3-r_2}{r_2}}\|\e_2\|.
\]
Then, after Lemma \ref{LemaPosV} from Appendix \ref{app:lemmata} with $A=\mu_m\gamma_2/2$, $B=\lambda_m\{\bGma^{-1}\}/2$, $C=k_{\Psi}^{(m_2-r_2)/r_2}\mu_M$, $d_0=d_{2}$, $r=r_2$, $s=r_2$, $m=m_2$, $y=\|\e_2\|$ and $x=\|\tbtta_U\|$, so if  $d_{2}$ satisfies \eqref{660b}, then $V_{nl}(\e_2,\tbtta_U) \geq 0$. It follows that $V_{A2}(\e,\tbtta_U)$ is positive definite and radially unbounded. Furthermore, long computations (see Appendix \ref{app:VA2}) establish the following. 

\begin{claim}\label{VA2negdef}
Let $\alpha_{A2} := \left(\gamma_3+\gamma_{4}\right)\gamma_{A2}d_{2}c_2$,  $p_1=3r_2/ r_1$, and $p_2=4r_2/m_c$. Then,
\begin{equation}
  \label{ap:163} 
\dot{V}_{A2}(\e,\tbtta_U) \leq -\alpha_{A2} \big[ k_1\|\e_1\|_{p_{1}}^{p_1} + k_2\|\e_2\|_3^{3} + z^{p_2} \big],
\end{equation}
for some $c_{2},k_1$, $k_2>0$, where $z:=\tanh(\|\z\|)$. 
\end{claim}

Integrating both sides of \eqref{ap:163} it follows that $\e_1\in {\mathcal L}_{p_1}^{n}$, $\e_2\in {\mathcal L}_3^{n},z\in{\mathcal L}_{p_2}$. Moreover, since $V_{A2}$ is positive definite and radially unbounded, $\e_1$, $\e_2$, $\tilde{\boldsymbol \theta} \in \mathcal L_{\infty}^{n}$.  Thus, invoking standard arguments from---see {\it e.g.,} \cite[Lemma 3.2.5]{IOASUN}, it  follows that $\e_1$, $\e_2$, and $z$ converge to zero asymptotically. The statement follows, considering that $\tanh$ is locally Lipschitz and $\tanh(0)= 0$. 
\end{proof}

\begin{remark}
Note that even though the convergence of the tracking errors and the product $\bPsi(\q)\tbtta_U$ follows from straight-forward arguments after \eqref{ap:163}, the difficulty resides in constructing the function $V_{A1}$ that satisfies \eqref{ap:163}. Indeed, the same conclusion cannot be withdrawn from the same ``standard arguments'' after the inequality \eqref{588}.
\end{remark}
  
Next, we present two statements establishing uniform finite-time stability under different conditions of persistency of excitation. In both cases, the proof is constructive; we provide a strict Lyapunov function for the closed-loop system, which, to the best of our knowledge, is the first in the literature of finite-time-stability of the origin under parametric uncertainty (not to be confounded with convergence to zero in finite time for part of the coordinates, as in Corollary \ref{cor:FTe} farther below).  

\begin{proposition}[Uniform short-FT stability]\label{prop:SFT}
Consider the system \eqref{476} under Assumption \ref{ass:gains}, with $\Delta$ as defined in \eqref{488} or \eqref{832}, and let \eqref{530} hold. Then, the origin is uniformly globally stable. If, in addition, for each $R>0$ there exist constants $\phi_R$, $w,T_{s}>0$, and $\ell\in(0,T_{s})$, such that, for any $\boldsymbol \xi_\circ$ such that $\|{\bf x}_\circ\|\leq R$,
\begin{equation}
\label{IELSDREM}
\int_{t}^{t+\ell}\zeta_{1}(\Delta(s, \boldsymbol \xi_\circ)){\rm d}s\geq w, 
 \qquad \forall t\in [0,T_{s}], 
  \quad  T_{s}\geq \ell+{2\ell\over w}\phi_R,
\end{equation}
then the origin is uniformly globally attractive in finite time. That is, for each $R$, there exists $T_R \leq T_s(R)$ such that ${\bf x}(t)\equiv 0$ for all $t\geq T_R$.
\end{proposition}
\begin{remark}
An important feature of this adaptive controller  is that  $\e_{1},\e_{2}$ and $\bPsi(\q)\tbtta_U$ go to zero asymptotically for all $2r_2 > r_1 > r_2$ even if the bound on $T_s$ in (\ref{IELSDREM}) does not hold. This is an improvement of composite adaptive controller in comparison with the indirect adaptive schemes, which is in absence of IE conditions always produce steady-state error. The constant $\phi_{R}$ depends on all initial error conditions and not solely on the bound of $\tbtta_{U}(0)$. 
\end{remark}
\begin{proof}
Consider the function $V_{A1}(\e,\tbtta_U)$ as defined in \eqref{SLFVA1}-\eqref{SLFV1}. Let Assumption \ref{ass:gains} generate  $r_1$, and $r_2$ such that $2r_2>r_1 >r_2$, and define $m_3 := 3r_2$ and the function 
\begin{equation}
\label{SLF6}
V_{FT}(\e,\tbtta_U) = (\gamma_3+\gamma_{4})\Big[V_{A1}(\e,\tbtta_U)^{\frac{m_3}{2r_2}} + V_{A1}(\e,\tbtta_U)^2\Big]
+\gamma_3d_{3}\lceil \e_1^{\top}\rfloor^{\frac{m_3-r_2}{r_1}}\M(\q)\e_2,
\end{equation}
where $\gamma_3	,\gamma_{4},d_{3}>0$, and 
\begin{equation}\label{d4value}
d_{3}^2\leq \frac{3\gamma_2^{3}\mu_m}{4\mu_M^2}\left(\frac{3r_1\Pm}{2r_2}\right)^{1/2}.
\end{equation}
Then, we underline the following statement, whose proof is provided in Appendix \ref{app:LemVft}.
\begin{claim}\label{LemVft}
If $\gamma_{A2}:=\gamma_3/(\gamma_3+\gamma_{4})$ is sufficiently small, 
there exist $\alpha_1$, $\alpha_2 \in \mathcal K_\infty$, such that 
\begin{equation}
  \label{742} \alpha_1(\|\x\|)\leq  V_{FT}(\e,\tbtta_U) \leq \alpha_2(\|\x\|).
\end{equation}
Moreover, there exists $c_3>0$, such that the time derivative of $V_{FT}$ along the trajectories of the closed-loop system (\ref{476}) satisfies 
\begin{equation}\label{dotVft}
\dot V_{FT}(\e,\tbtta_U) \leq  -\gamma_3c_3\zeta_{1}(\Delta(t))\Big[  
V_{FT}(\e,\tbtta_U)^{\sigma_1} + V_{FT}(\e,\tbtta_U)^{\sigma_2} \Big], 
  \quad \forall \x\in\mathbb{R}^{2n+\jmath},
\end{equation}
where $\sigma_1 := (4r_2 - r_1)/m_3$ and $\sigma_2 := 2m_2/3r_2$ both satisfy $\sigma_1$, $\sigma_2<1$. 
\end{claim}
Next, let  and $v_{FT}(t):= V_{FT}(\e(t),\tbtta_U(t))$. After \eqref{dotVft}, we have that $\dot v_{FT}(t) \leq -\gamma_3c_3\zeta_{1}(\Delta(t,\bxi_\circ))v_{FT}^{\sigma_1}$ and, also, $\dot v_{FT}(t) \leq -\gamma_3c_3\zeta_{1}(\Delta(t,\bxi_\circ))v_{FT}^{\sigma_2}$, for all $t\in [0,T_s]$.  Also, after \cite{WaEfBo2021}, we have that if \eqref{IELSDREM} holds for each fixed $\bxi_\circ$, then 
\begin{equation}\label{PElowerboundz2}
\begin{aligned}
	\int_{t_0}^{t_0+t} \zeta_{1}(\Delta(s, \boldsymbol \xi_\circ)) \, {\rm d}\sigma   \geq \frac{w}{2 \ell} (t-\ell) \quad \forall\, t\in [0,T_s] 
\end{aligned}	
\end{equation}
Using the latter, the comparison lemma, and the inequality $v_{FT}(0)\leq \alpha_2(\|\x_\circ\|)$,  we obtain 
\begin{equation}
  \label{749} 
v_{FT}(t)\leq \min\left\{ \Big[ \alpha_2(\|\x_\circ\|)^{r_1-r_2\over 3r_2} - \frac{\rho}{3r_2}[t-\ell]  \Big]^{3r_2\over r_1-r_2}, \ \Big[ \alpha_2(\|\x_\circ\|)^{r_1-r_2\over 2m_2} - \frac{\rho}{2m_2}[t-\ell] \Big]^{2m_2\over r_1-r_2} \right\}, 
\end{equation}
where $\rho:= w(r_1-r_2)\gamma_3c_3/ 2\ell$. From the latter and \eqref{742}, it follows that, for any $R>0$ and $\|\xi_\circ\|\leq R$, there exists $T_R \in [0,T_s]$ such that $v_{FT}(t) = 0$ for all $t\geq T_R$, provided that
\begin{equation}\label{Tf1}
T_{s} \geq \ell+ \frac{2\ell }{w} \frac{\phi_{R}}{(r_1-r_2)\gamma_3c_3}
\end{equation}
where 
\[
\phi_{R} := \max\left\{ 3r_2 \alpha_2(R)^{r_1-r_2\over 3r_2}, \quad 
    [4r_2+2r_1]\alpha_2(R)^{r_1-r_2\over 4r_2+2r_1} \right\}. 	
\]
\end{proof}
\begin{remark}
The bound on the settling time is uniform in the size of the initial conditions and in the initial time---see Remark \ref{rmk:t0}, but it is conservative. Note that the bounds on $V_{FT}$ in \eqref{742} are uniform so one could replace $\alpha_2(R)$ by $\alpha_2(R_x)$ where $R_x$ is a bound on $\|x_\circ\|$. However, we write the bound \eqref{Tf1} for simplicity since, in general, $\ell$ and $w$ depend on $R$, as per the design of the adaptive controller and the explanations provided in  Remark \ref{rmk:t0}. 
\end{remark}

\begin{proposition}[UFT under IE]\label{prop:UFT-IE}
Consider the system \eqref{476}, with $\Delta$ as defined in \eqref{488}, under Assumption \ref{ass:gains}, and let \eqref{530} hold. Then, the origin is uniformly globally stable. In addition, for each $R>0$, let there exist constants $\mu_R$ and $T_R>0$ such that \eqref{502} holds for all $\|\bxi_\circ\|\leq R$, for $\Omega$ as defined as in \eqref{361} or \eqref{419}. Then, the origin is uniformly globally attractive in finite time. That is, for each $R$, there exists $T_R > 0$ such that ${\bf x}(t)\equiv 0$ for all $t\geq T_R$. 
\end{proposition}
\begin{proof}
Consider the function $V_{FT}$ defined in \eqref{SLF6}. Its total derivative along the trajectories of the closed-loop system satisfies \eqref{dotVft} for all $t\geq 0$. Then, in view of the assumption of IE on $\bOmga$, \eqref{eq8} holds. That is, there exits $\phi_m > 0$ such that $\zeta_{1}(\Delta(t,\bxi_\circ))\geq \phi_m$, for all $t\geq T_{IE}$ and all $\|\x_\circ\|\leq R$. Therefore, 
\begin{equation}
  \label{776} 
\dot v_{FT}(t) \leq  -\gamma_3c_3\phi_mv_{FT}(t)^{4r_2-r_1\over m_3}, \qquad \quad \forall t\geq T_{IE}. 
\end{equation}
Integrating on both sides of the previous inequality and applying the comparison lemma, we obtain
\[
v_{FT}(t)\leq \Big[v_{FT}(0)^{r_1-r_2\over 3r_2} - {(r_1-r_2)\gamma_3c_3\over 3r_2}\phi_{m}(t-T_{IE})\Big]^{3r_2\over r_1-r_2}, \quad \forall t\geq 0.
\] 
We conclude that $v_{FT}$ converges to zero in finite time, with maximum settling-time 
\[
T_0 \leq {3r_2 \alpha_2(R)^{r_1-r_2\over 3r_2}\over (r_1-r_2) \gamma_3c_3\phi_{m}}+T_{IE}.
\] 
Thus, in general $T_0$ depends only on the size of the initial conditions
\end{proof}

\begin{remark}
For $\q_d$ corresponding to natural equilibria, finite-time convergence is always guaranteed, independently of whether the unknown parameter is estimated or not, for any initial condition on position and velocity. However, when $\q_{d}$ is not natural equilibrium, FT regulation is only ensured under the conditions of Proposition \ref{prop:SFT} or Proposition \ref{prop:UFT-IE}.
\end{remark}

The following statement on finite-time stability stands for the case in which the parameters are known, or the parameter estimation errors converge in finite time. Even though it follows as a corollary of our main results, to the best of our knowledge, it is the first of its kind in the literature.

\begin{corollary}\label{cor:FTe}
Consider the system \eqref{476} under Assumption \ref{ass:gains} with arbitrary  initial conditions $\x_\circ \in \mathbb R^{2n+\jmath}$ and let \eqref{530} hold.  If there exists $T_\theta >0$ such that $\bPsi(\q(t))\tbtta_U(t) = 0$, for all $t\geq  T_\theta$, then there exists $T_e\geq T_\theta$ such that $\e_{1}(t) \equiv \e_{2}(t) \equiv 0$, for all $t\geq T_e$.
\end{corollary}
\begin{proof}
The function $W_{FT}(\e) = V_{FT}(\e,0)$, with $V_{FT}$ as defined  in (\ref{SLF6}), is a Lyapunov function for the system
\begin{equation}\label{CLpendulo0}
\begin{array}{rl}
		\dot \e_1=&\e_2\\
		\dot \e_2=& \M^{-1}(\q)\left[-\C(\q,\e_2)\e_2-\P\lceil \e_1\rfloor^{a}-\D\lceil \e_2\rfloor^{b}-\D_l\e_2\right].
\end{array}
\end{equation}
Moreover, there exist  $\gamma_3$ and $c_3>0$ such that the time derivative of $W_{FT}$ along the closed-loop system (\ref{CLpendulo0}) satisfies 
\begin{equation}\label{dotVftn}
\dot W_{FT}(e) \leq  -\gamma_3c_3 \Big[ W_{FT}(e)^{\sigma_1} + W_{FT}(e)^{\sigma_2} \Big].
\end{equation}
where $\sigma_1 := {(4r_2-r_1)/m_3}$ and $\sigma_2:= {2m_2/3r_2}$. The proof of the latter claim follows from the proof of Claim \ref{LemVft}---specifically, from the expression \eqref{TDofV3ub}. This and the fact that there exists $T_\theta$ such that for all $t\geq T_\theta$, $\bPsi(\q(t))\tbtta_U(t) = 0$, implies that the dynamic equation (\ref{476}) reduces to (\ref{CLpendulo0}). That is, for this system, $W_{FT}(\e)$ is a strict Lyapunov function whose time derivative along the trajectories of (\ref{CLpendulo0}) satisfies (\ref{dotVftn}). Since $\sigma_1<1$ and $\sigma_2<1$, finite-time stability follows. 
\end{proof}

\section{Simulation Examples}\label{SecSimExa}

In this section we present some numerical simulations performed using the model of a 2-DoF robot manipulator with revolute joints given by \eqref{237} with 
$$
{\M}=
\begin{bmatrix}
  \delta_{1}  + 2\delta_{2}{\rm c}_{2} & \ast\\
   \delta_{3}  + \delta_{2}{\rm c}_{2} & \delta_{3}
\end{bmatrix}, \quad
\C=\begin{bmatrix}
  -\delta_{2}{\rm s}_{2}{\dot q}_{2} & -\delta_{2}{\rm s}_{2}({\dot q}_{1} +{\dot q}_{2})\\
   \delta_{2}{\rm s}_{2}{\dot q}_{1} & 0
\end{bmatrix}, \quad \nabla_{\q}{\mathcal U}({\bf q}) =
\begin{bmatrix}
 \delta_{4}{\rm s}_{12} + \delta_{5} {\rm s}_{1}\\
  \delta_{4}{\rm s}_{12}
\end{bmatrix},
$$
where $\delta_{1}:= (l_{1}^{2}+l_{c2}^2 )m_{2} + l_{c1}^2 m_{1} + I_{1}+I_{2}$, $\delta_{2}:= l_{1} l_{c2} m_{2}$, $\delta_{3}:= l_{c2}^2 m_{2}+I_{2}$, $\delta_{4}=m_{2}l_{c2}g$, and $\delta_{5}=(m_{1}l_{c1}+m_{2}l_{1})g$. The lengths, masses, and inertias of each link are denoted by $l_i$, $m_i$, and $I_i$, respectively, for $i\in{1,2}$. The physical parameters are set to $m_{1}=2$ kg, $m_{2}=1$ kg, $l_{1}=0.3$ m, $l_{2}=0.2$ m, and $g=9.81$ m/s$^{2}$. Also, for brevity, we use ${\rm c}_{2}=\cos(q_{2})$, ${\rm s}_{2}=\sin(q_{2})$, and ${\rm s}_{12}=\sin(q_{1}+q_{2})$.

For comparison, we performed simulations using different parameterizations, as well as the controllers proposed in \cite{NaMaHeReBa2015} and \cite{SlotLee1989}. Some numerical tests were conducted with and without friction torques $\t_f \in \mathbb{R}^2$ and measurement noise in both position and velocity. The results are presented in two separate subsections. For proper comparison, in all the tests the desired joint position is $\q_{d}=[2,2]^{\top}$, with initial conditions $\q(0)=[3,0]^{\top}$ and $\dq(0)=[0,0]^{\top}$, and all the simulations are carried out using Euler discretization with  fixed step $T_{s}=5\times10^{-4}$ s. For further reference, we define the following. 

\noindent\textbf{Controller C1}: We used the control input as in \eqref{278} with $\P=3\I$, $\D=2\I$, $r_{1}=1.5$ and $r_{2}=1$. For the adaptation law \eqref{318} we used $\gamma_{1}=0.3$, $\gamma_{2}=0.7$, $d_{1}=5$, $\bGma=\I_{2}$, $\bUps_{U}=50\I_{2}$, and to generate $\Delta$ we used the force-balance parameterization and the least-squares DRE; more precisely, \eqref{488}  and \eqref{443} with $\alpha=\xi=\beta=10$, $f_0 = 1$, and ${\hat \vro}_0 = 0$. 

\noindent\textbf{Controller C2}: is defined exactly as the controller C1, but in place of the least-squares DRE we used Kreisselmeier's dynamic regressor extension to generate $\Delta$ with filter parameters  $\lambda_{2}=1$, and $\lambda_{3}=1.3$ in \eqref{510}.

\noindent{\bf Controller C3}: is taken from  \cite{NaMaHeReBa2015}.   This  was originally designed for trajectory-tracking tasks but here we applied it for regulation. In this case, the controller is given by
\[
\t=\begin{cases}
\t_{1}     &\hbox{if} \quad f(\e_{1},\e_{2})\leq 0,\\
\t_{2}	  & \hbox{if} \quad f(\e_{1},\e_{2})>0,\\
\end{cases}
\]
where
\begin{equation}
  \label{795} 
f(\e_{1},\e_{2}):= \e_{2}^{\top}{\bf M}(\q)\e_{2}-\lambda_{M}\{{\bf M}(\q)\}(K_{2}\lceil \e_{1}\rfloor ^{a})^{\top}(K_{2}\lceil \e_{1}\rfloor ^{a}), 
\end{equation}
with $a$ and $K_{2}$ are as above, and 
\begin{align*}
	\t_{1}=& W(\q,\dq,\dq_{r},\ddq_{r})\hbtta-{K}_{1}{\bf s}-u_{r}, \quad K_{1}={2}, \quad K_{s}={0.6*\I},\\
	u_{r}:=&\begin{cases}
		K_{s}{{\bf s}\over \|{\bf s}\|}  &\hbox{if} \qquad 	{\bf s}\neq{\bf 0},\\
		{\bf 0}	&\hbox{if} \qquad 	{\bf s}={\bf 0},
	\end{cases}
\end{align*}
with ${\bf s}=\dq+{K}_{2}\lceil \tq\rfloor ^{a}$, $
		\dq_{r}=-{K}_{2}\lceil  \tq\rfloor ^{a}$, $
		\ddq_{r}=-a{K}_{2}{\rm diag (|\tq|}^{a-1})\dq$ and 
\begin{subequations}\label{stsmvar}
		\begin{align}
		W(\cdot)=&\begin{bmatrix}
			\ddq_{r1} & W_{12}& \ddq_{r2} & s_{12} &  s_{1}\\
			0       & W_{21} & \ddq_{r1}+ \ddq_{r2}&  s_{12}&0
		\end{bmatrix}\label{defW}
	\end{align}
\end{subequations}
and $W_{12}=\cos(q_{2})(2\ddot q_{r1} +\ddot q_{r2})-\sin(q_{2})[\dot q_{2}\dot q_{r1} +(\dot q_{1} + \dot q_{2})\dot q_{r2}]$ ,$W_{21}=\cos(q_{2}){\ddot q}_{r1} + \sin(q_{2})\dot q_{1}\dot q_{r1}$, $K_{1}=2$ and $K_{2}=3/2$, and the parameter adaptation law is 
\begin{equation}
  \label{63} \dot \hbtta=-\Gamma W^{\top}(\cdot) {\bf s}-\Gamma k\bPhi_2^{\top} {\Theta\over \|\Theta\|}, \quad \Theta=\bPhi_2{\hbtta} -\bPhi_1,  
\end{equation}
with $\Gamma=0.001$ and $k=5000$. On the other hand, 
\begin{align*}
\t_{2}=&\t_{s}-u_{r}, \quad K_{s}={0.6*\I},\\
u_{r}:=&\begin{cases}
		K_{s}{{\bf s}\over \|{\bf s}\|}  &\hbox{if} \qquad 	{\bf s}\neq{\bf 0},\\
		{\bf 0}s	&\hbox{if} \qquad 	{\bf s}={\bf 0},
	\end{cases}	
\end{align*}
with  
\begin{subequations}\label{AsymptoticAC}
	\begin{align}
	\t_{s}:=&W(\q,\dq,\dq_{r},\ddq_{r})\hbtta-{K}_{1}{\bf s},\\
	{\bf s}=&\dq+{K}_{2}\tq, \label{vrSlotineLia}
\\
		\dq_{r}=&-{K}_{2}\tq, \label{vrSlotineLib}\\
		\ddq_{r}=&		-{	K}_{2}\dq, \label{vrSlotineLic}
	\end{align}
\end{subequations}
and $W(\cdot)$ with the structure given by (\ref{defW}) but using the variables (\ref{vrSlotineLia})-(\ref{vrSlotineLic}), and
\[
\dot \hbtta = -\Gamma W^{\top}(\cdot) {\bf s}-\Gamma k\Theta, \quad \Theta = \bPhi_2 \hbtta -\bPhi_1, \quad \Gamma=1, k=50.
\] 
In this case, $\bPhi_1$ and $\bPhi_2$ are obtained from (\ref{510}) with $\lambda_{3}=1$, and $\hbtta\in \mathbb{R}^{5}$.

\noindent{\bf Controller C4}: It is taken from \cite{SlotLee1989} and it is given by (\ref{AsymptoticAC}), but the parameter adaptation law is now given by
\begin{subequations}\label{estimadorDREMorig}
	\begin{align}
			\dot \hbtta=&-{\bf P}(t)\left[ W^{\top}(\cdot) {\bf s}+\bOmga\e_{p}\right], \quad \e_{p}=\bOmga{\hbtta} -\y,\label{LSSlotineLi}	\\ 
		{\dot\P} =& -\alpha \P \bOmga^{\top} \bOmga\P + \beta \P, \qquad \P(0)={1\over p_{0}}\I_{5}, \label{LSorginal}\\
\beta :=& \beta_0 \left(1 - \frac{\|\P\|}{k_{0}}\right),\end{align}
\end{subequations}
where $p_{0}=f_{0}$, $k_{0}=\xi$, and $\y$ is obtained from (\ref{NLRPE}). Again, this adaptive control scheme requires the estimate of all the system parameters. 

 For a fair comparison, the proportional and derivative gains--identical for controllers \textbf{C1} and \textbf{C2}--are chosen equal to $K_{1}$ and $K_{1}K_{2}$, respectively, as in controllers \textbf{C3} and \textbf{C4}. Furthermore, the filtered regressor matrix $\bOmga$ is constructed using the force-balance parameterization, {\it i.e.,} as in \eqref{419}. 

\subsection{Case 1: Tests without friction or measurement noise}

The results with no friction and noise in the position and velocity measurements are shown in figures \ref{fig:Ejemplo2}  and \ref{fig:Ejemplo3}.

The controllers {\bf C1} to {\bf C3} drive the regulation error to zero in finite time, approximately at 4 [s], as shown in Fig. \ref{fig:Ejemplo2}(a)-(b). In contrast, controller {\bf C4} yields asymptotic convergence of the regulation error to zero, see Fig. \ref{fig:Ejemplo2}(a)-(b). Observe that controllers {\bf C2} and {\bf C3} exhibit very similar performance, as shown in Fig. \ref{fig:Ejemplo2}(b) and (d), although {\bf C2} employs a simpler regressor than {\bf C3}. Under the proposed gain tuning, the performance of the second joint is very similar for all controllers, and no steady-state errors are observed. With respect to transient behavior, controller {\bf C4} provides a monotonic transient response, particularly for the first joint. This notable behavior may be attributed to the fact that the time-varying gain $\P(t)$ multiplies the estimator terms of the {\bf C4} controller, whereas in the other controllers the estimators employ only constant gains. In particular, for controller {\bf C1}, the constant estimator gain is associated with a design based on LS with a forgetting factor combined with DREM technique. The introduction of other estimators not based on DREM may further improve transient performance by allowing time-varying gains in both estimator terms.
 
 The non Lipschitz terms of the PD part of controllers {\bf C2} and {\bf C3} are responsible for finite-time convergence, see figure \ref{fig:Ejemplo2}(a)-(b). However,  these terms also introduce slight chattering in the control signals due to discretization errors, as shown in figure \ref{fig:Ejemplo2}(e) and \ref{fig:Ejemplo2}(f).	The amplitude of this chattering  is not greater than $0.015$[Nm], which is approximately $100$ orders of magnitude smaller than the torque required to maintain the links at the desired positions. In contrast, when using controller {\bf C4} controller, the control signal exhibits significant chattering level, see Figure \ref{fig:Ejemplo2}(e) and \ref{fig:Ejemplo2}(f). 
 
 \begin{figure}[h!]
\begin{center}
\includegraphics[width=12cm]{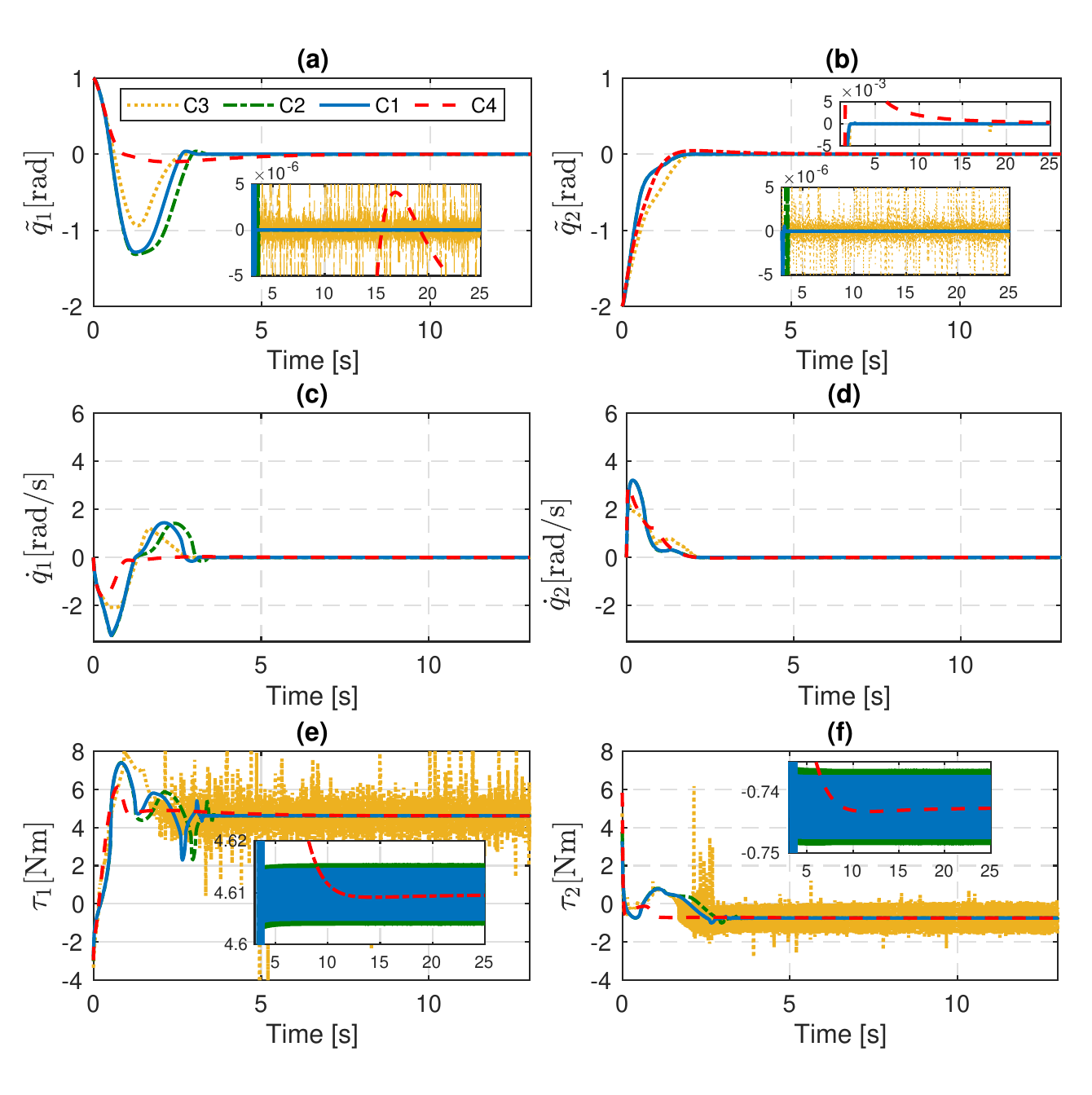}    
\vspace{-.3cm}
\caption{Simulation results of the 2-DOF-EL-system. The first column shows the position error, velocity and torque of the first link while the second one shows the position error, velocity and torque of the second link. {\bf C1} is in solid line, {\bf C2} is in dashed-dotted line, {\bf C3} is in dotted line, and {\bf C4} is in dashed line.} 
\label{fig:Ejemplo2}
\end{center}
\end{figure}
 
 With the proposed estimation laws, the parameters converge to zero no later than $2.5$[s], see figure \ref{fig:Ejemplo3}(a) and \ref{fig:Ejemplo3}(d). This ensures that  $\bPsi(\e_{1}+\q_{d})\tbtta_{U}$ becomes zero in finite time, since the parameter estimation errors also vanish in finite time, see Figure \ref{fig:Ejemplo3}(b) and \ref{fig:Ejemplo3}(e). In controller {\bf C3}, the parameters are estimated in finite time, but at the cost of increased chattering due to the use of a discontinuous estimator. Figure \ref{fig:Ejemplo3}(c), \ref{fig:Ejemplo3}(f) show the evolution of the integral term $\int_{0}^{t}\zeta_{1}(\Delta(s))$d$s$, while figure \ref{fig:Ejemplo3}(i) shows the minimum eigenvalue of the filtered regressor $\bPhi$ involved in the parameter estimator of {\bf C3}. Figure \ref{fig:Ejemplo3}(l) depicts the minimum eigenvalue of the matrix on the left-hand side of \eqref{502}. For the estimators in {\bf C2},  and  {\bf C4}, the plots shows that the integrals tend to a constant, while for the estimator in  {\bf C1}, it tends to infinity. This illustrate that, in this later case, the LS estimator produces persistence excitation. In fact, for the {\bf C4}, figure \ref{fig:Ejemplo3}(l) shows that the LS estimator (\ref{LSorginal}) provides persistence excitation for $\bOmga$ after a certain time instant. For {\bf C2} and {\bf C3}, $\zeta_{1}(\Delta(t))$  and the function $\bPhi_2$ are not PE since the desired reference is constant, see figure \ref{fig:Ejemplo3}(f) and (i).  However, figure \ref{fig:Ejemplo2}(i) shows that,  as long as $\lambda_{m}\{\bPhi_2\}$ is positive, it is sufficient to estimate the parameters, as depicted in Figure \ref{fig:Ejemplo2}(g). 
\begin{figure}[h!]
\begin{center}
\includegraphics[width=14cm]{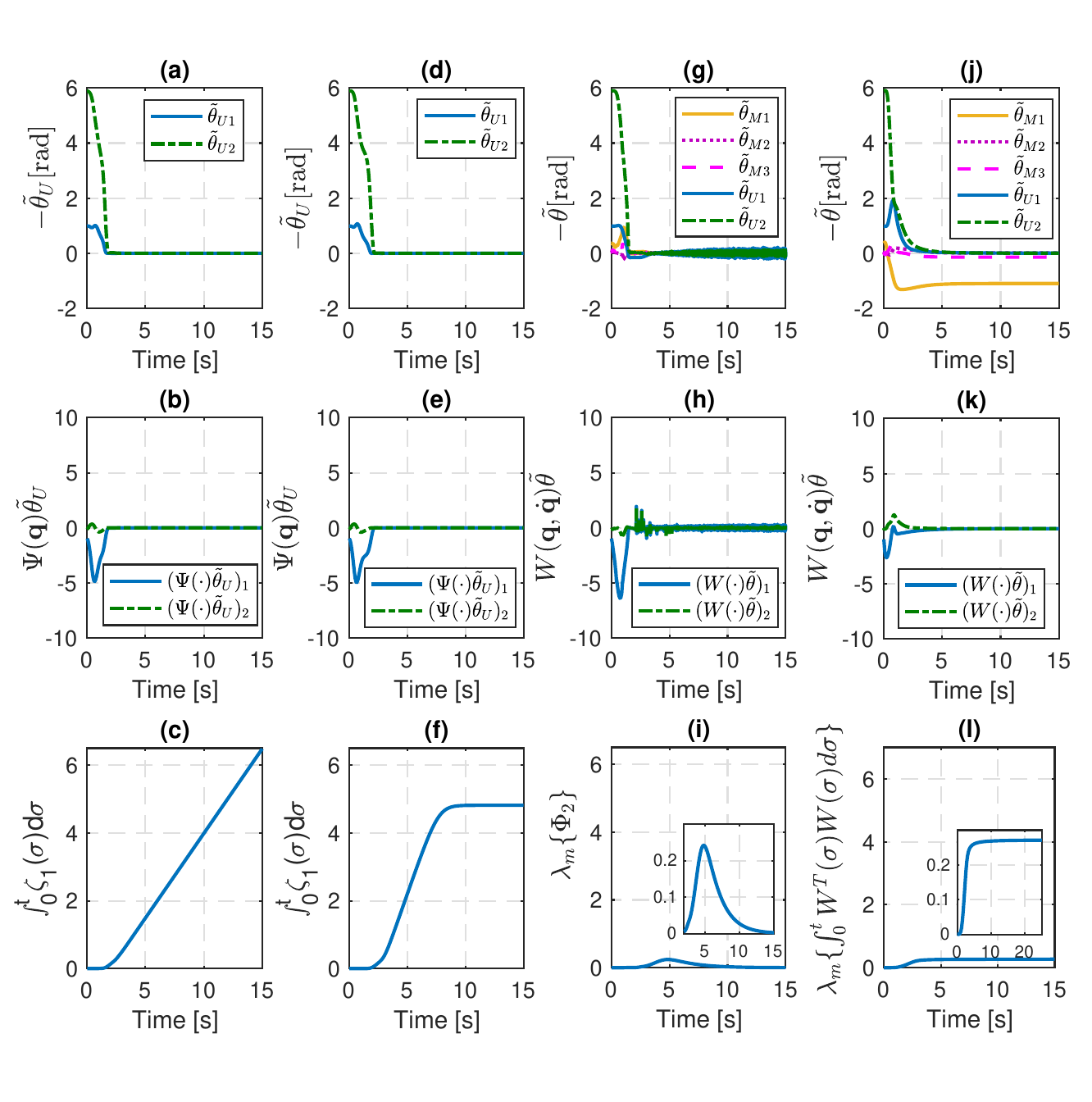}    
\vspace{-.5cm}
\caption{Simulation results of the 2-DOF-EL-system. The first column shows the plots for {\bf C1}, the second column shows the plots for {\bf C2}, the third column shows the plots for {\bf C3} and the fourth column shows the plots for {\bf C4}.} 
\label{fig:Ejemplo3}
\end{center}
\end{figure}


\subsection{Case 2: With friction and measurement noise} 

For the second test, we use the same data as in Example 1 and, in addition, we consider Coulomb friction torques and measurement noise. The friction is modeled as
\[
\t_{f}:=\begin{bmatrix}
	0.5{\rm sign}(\dot q_{1})\\
	0.4{\rm sign}(\dot q_{2})
\end{bmatrix},
\]
while the noise affecting position and velocity measurements is modeled as
 \[
 {\bf n}_{\q}=0.005\times \begin{bmatrix}
 	\sin(100t) \\
 	\cos(100t) 
 \end{bmatrix},\quad 
 {\bf n}_{\dq}=0.005\times \begin{bmatrix}
 	\sin(100t) \\
 	\sin(100t) 
 \end{bmatrix}.
\]

The results when friction and noise is added to the position and velocity measurements are shown in Figures \ref{fig:Ejemplo4} and \ref{fig:Ejemplo5}.

 The behavior in terms of position and velocity errors is very similar to the previous case. However, steady-state error appears. For the finite-time control schemes, this error is  smaller in magnitud than the asymptotic controller {\bf C4}, as depicted in figures \ref{fig:Ejemplo4}(a) and (b). This difference  is more evidente in the first joint, as shown in figure \ref{fig:Ejemplo4}(a). Note that the control signals of {\bf  C1} and {\bf C2} are smooth during the transient response,  but chattering is significantly increased y steady-state response. This contrasts with the indirect finite-time adaptive controllers proposed in \cite{CuzMoToA2025}, where the chattering does not  increased as much. This behavior can be explained by the fact that the composite adaptive schemes now use feedback of the position errors and velocities, which are affected by measurement noise. However, in order to achieve the same steady-state error, controller {\bf C4} must increase its gains, resulting in signals with a higher noise magnitude. Therefore, for all the controllers, in presence of unmodeled affects and noise, there is a trade-off between small steady-state error and the magnitud of noise in the control signal. One way to reduce chattering in controllers {\bf C1} and {\bf C2} is to make the PD gains smaller, at the cost of slower convergence and a larger steady-state error. 
\begin{figure}[h!]
\begin{center}
\includegraphics[width=12cm]{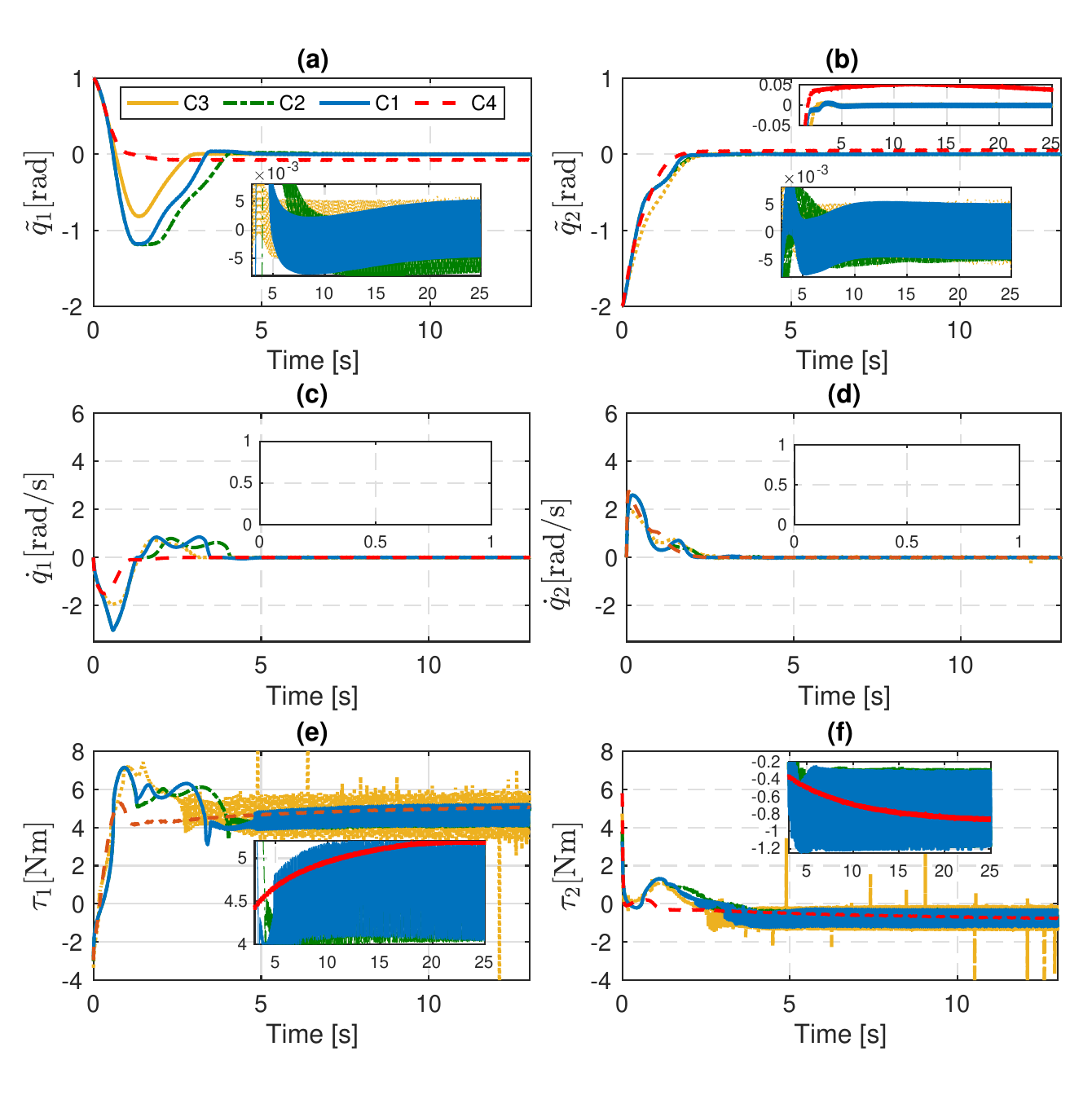}    
\vspace{-.3cm}
\caption{Simulation results of the 2-DOF-EL-system. The first column shows the position error, velocity and torque of the first link while the second one shows the position error, velocity and torque of the second link. {\bf C1} is in solid line, {\bf C2} is in dashed-dotted line, {\bf C3} is in dotted line, and {\bf C4} is in dashed line.} 
\label{fig:Ejemplo4}
\end{center}
\end{figure}
  
With respect to parameter estimation, Fig. \ref{fig:Ejemplo5} shows that the unmodeled effect and measurement noise considerably degrade the convergence of the parametric error, particularly for the {\bf C4} controller, as shown in Fig. \ref{fig:Ejemplo5}(j). This behavior can be explained by the fact that these effects significantly modify the persistency of interval excitation of the system (see Fig. \ref{fig:Ejemplo5}(c), \ref{fig:Ejemplo5}(f), and \ref{fig:Ejemplo5}(i)), which in turn has a substantial impact on parameter estimation, as illustrated in Fig. \ref{fig:Ejemplo5}(a), \ref{fig:Ejemplo5}(d), \ref{fig:Ejemplo5}(g), and \ref{fig:Ejemplo5}(j). The {\bf C4} controller can partially compensate for the error caused by poor parameter estimation through the use of discontinuous functions. However, this improvement comes at the cost of inducing significant chattering effects. In particular, for the first joint, as depicted in Figure \ref{fig:Ejemplo4}(e). For the other adaptation laws, only steady-state boundedness of the parameter error can be guaranteed. This results in errors in the controller terms involving the regressors, as shown in Fig. \ref{fig:Ejemplo5}(b), \ref{fig:Ejemplo5}(e), and \ref{fig:Ejemplo5}(k). This source of error is responsible for the steady-state error observed in the position variable.

\begin{figure}[h!]
\begin{center}
\includegraphics[width=14cm]{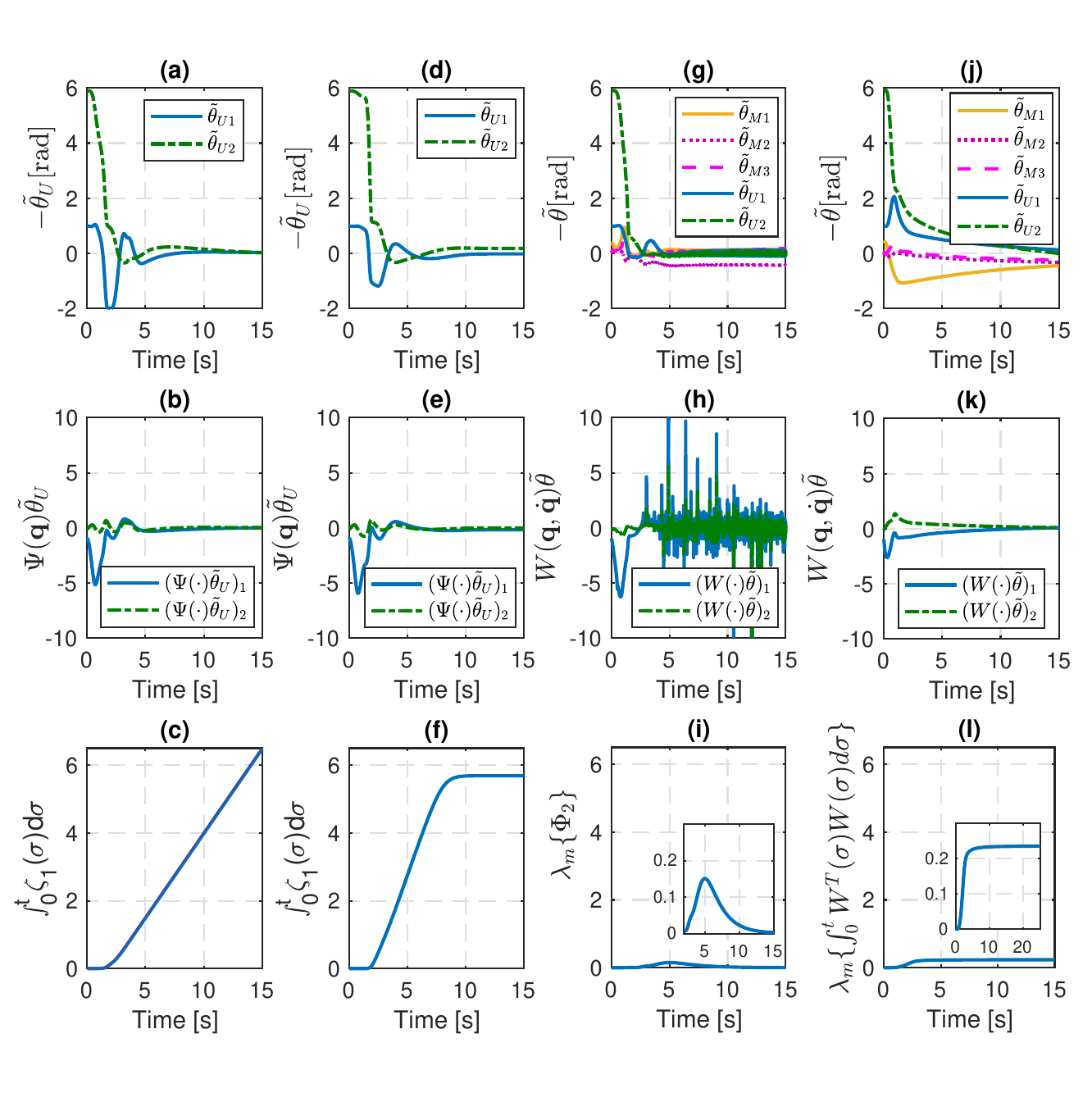}    
\vspace{-.5cm}
\caption{Simulation results of the 2-DOF-EL-system. The first column shows the plots for {\bf C1}, the second column shows the plots for {\bf C2}, the third column shows the plots for {\bf C3} and the fourth column shows the plots for {\bf C4}.} 
\label{fig:Ejemplo5}
\end{center}
\end{figure}

\clearpage
\section{Conclusions}\label{SecConcl}

This work focused on the development of adaptive control schemes to solve the FT regulation problem without exact knowledge of the gravity vector. The problem is addressed from composite adaptive control perspectives. For the purpose of designing FT estimation laws, two parameterizations and two dynamic regressor extensions are employed, based on recent DREM techniques. On the other hand, composite adaptive FT control without relying on nonlinear virtual reference signals, or terminal sliding mode techniques, has not been previously developed. In this work, we propose a solution inspired by \cite{Tomei1991} and DREM-based estimators. The composite adaptive controller guarantees asymptotic convergence even in the absence of excitation, albeit at the cost of a more complex estimator structure. Parameter estimation is achieved in finite time provided that sufficient excitation over a finite interval of time. This, in turn, guarantees that the regulation error converges to zero in finite time over that interval. Future research directions include extending adaptive FT control to 
$n$-DOF Euler--Lagrange systems with bounded control torques and without relying on velocity measurements.

\section*{Acknowledgements}

This work has been partially supported by the Mexican SECIHTI through Frontier and Basic Scientific Research grant CBF2023-2024-1964, Project supported by SECIHTI in 2025; 
and by the PAPIIT-UNAM grant number IN106323.

\appendix

\section{Proofs details}\label{SecProofdetail}
\subsection{On the function $\dot V_{1}$ in (\ref{588}) }
\label{app:dotV1}
Taking into account the facts that $\dot \M-2\C$ is a skew-symmetric matrix and that $\dot \M=\C^{\top}+\C$, the time derivative of $V_{1}$ along the trajectories of \eqref{476}
satisfies 
\[
\begin{array}{rl}
{\dot V_1(\e,\tbtta_U)\over \gamma_{1}+\gamma_{2}}=& {\dot V}_0+ \gamma_{A1}d_1\e_2^{\top} \mbox{Sech}(\e_1)^2\M(\q)\e_2+\gamma_{A1} d_{1}\tanh(\e_1)^{\top}\C^{\top}(\q,\e_2)\e_2,\\
 &-\gamma_{A1} d_{1}\tanh(\e_1)^{\top}\big[\P\lceil\e_1\rfloor^{a}+\D\lceil \e_2\rfloor^{b}
-\bPsi(\q)\tbtta_U\big]\\
&-\tbtta_U^{\top}\bPsi^{\top}(\q)[-\gamma_{A1} d_{1}\tanh(\e_1)-\dq]-\tbtta_U^{\top}\bGma^{-1}\dot \tbtta_U.
\end{array}
\]
In the equation above, $\gamma_{A1} = \gamma_1/(\gamma_1 + \gamma_2)$, ${\rm diag}				\{{\rm sech}^{2}(e_{1i})\}_{i=1}^{n}:=\mbox{Sech}^{2}(e_{1})\in{\mathbb R}^{n\times n}$, and 
$
{\dot V}_0
 =-\e_2^{\top}\left(\D\lceil \e_2\rfloor^{b}+\D_L\e_2-\bPsi(\q)\tbtta_U\right)
$.

After some simplifications, we obtain that $
 \dot V_1(\e,\tbtta)=-v(\e)-\zeta_{1}\left(\Delta\right)\tbtta_U^{\top}\bUps_{U}\left\lceil \tbtta_U\right\rfloor ^{b}$,
where
\begin{align}
v(\e):=&\e_2^{\top}\left(\D\lceil \e_2\rfloor^{b}+\D_L\e_2\right)
 -\gamma_{A1} d_{1}\e_2^{\top} \mbox{Sech}(\e_1)^2\M(\q)\e_2,\nonumber\\
 &-\gamma_{A1} d_{1}[\tanh(\e_1)]^{\top}\C^{\top}(\q,\e_2)\e_2-\gamma_{A1} d_{1}[\tanh(\e_1)]^{\top}\big[\P\lceil\e_1\rfloor^{a}+\D\lceil \e_2\rfloor^{b}\big].\label{defvtta}
\end{align}
Recall that $\| \tanh(\e_1)\|\leq \sqrt{n}$, $\| \tanh(\e_1)\|\leq {n}^{1-c\over 2}\| \tanh(\e_1)\|^{c}$, for some constant $1\geq c\geq 0$, and and $\|\mbox{Sech}^{2}(e_{1})\|\leq 1$. From these facts, and Properties (P1)-(P3) together with Lemmata \ref{lemaMeanP}-\ref{lemaMeanP0}, we find the following upper bounds:
\begin{subequations}\label{BoundCrosstermsAS}
\begin{align}
\e_2^{\top}\mbox{ Sech}(\e_1)^{2}\M(\q)\e_2\leq& \mu_{M}\|\e_2\|^{2},\label{BoundCrossT1}\\
 \tanh(\e_1)^{\top} {C}^{\top}(\q,\e_2)\e_2\leq& \sqrt{n}L_{c}
\|\e_2\|^{2},\\
  \tanh(\e_1)^{\top}{ D}\lceil \e_2\rfloor^{b} \leq& A_{0}\|\tanh(\e_1)\|\|\e_2\|^{b},\nonumber\\
  \leq & A_{0}n^{b-a\over 2(1+b)}\|\tanh(\e_1)\|^{1+a\over 1+b}\|\e_2\|^{b},\\
p_{m}\|\tanh(\e_1)\|^{1+a}\leq p_{m}\sum_{i=1}^{n}\left[\tanh(e_{1i})\right]^{1+a} \leq &\tanh(\e_1)^{\top}\P \lceil\e_1\rfloor^{a}, \label{Boundnormtanh}\\
\dm\|\e_2\|^{1+b} \leq &\e_2^{\top}\D\lceil \e_2\rfloor^{b},\\
   \lambda_{m}\{\bUps_{U}\}\|\tbtta_{U}\|^{1+b} \leq &   \tbtta_{U}^{\top}\bUps_{U}\left\lceil \tbtta_{U}\right\rfloor ^{b} .
\end{align}	
\end{subequations}

Here, we have defined $A_{0}:=n^{(1-b)/2}d_{M}$. Taking into account all the previous bounds and gathering terms, we arrive at (\ref{588}).

\subsection{On the function $V_{A1}$ in (\ref{SLFVA1}) }
\label{app:VA1}

The function $V_{A1}$ in \eqref{SLFVA1} satisfies
\begin{align}
V_{A1}\left(\e,\tbtta\right) &= \gamma_2V_0+\gamma_1{r_1\over 2r_2}\e_1^{\top}\P\lceil \e_1\rfloor^{a}+\frac{1}{2}\tbtta^{\top}\bGma^{-1}\tbtta+\gamma_1^{-{r_1\over 2r_2}}V_1^{m_2\over 2r_2}\left(\e,\tbtta\right)+\gamma_1V_{nn1}(\e,\tbtta),\nonumber \\
V_{nn1}\left(\e\right) & :={1\over 2}\e_2^{\top}\M(\q)\e_2 +d_{1}[\tanh(\e_1)]^{\top}\M(\q)\e_2+d_{1}\sum_{i=1}^{n}D_{Li}\ln(\cosh(e_{1i})).
\end{align}
Clearly, since $V_0>0$, it follows that $V_{A1}>0$ if $V_{nn1}\geq0$. The latter inequality follows since $V_{nn1}\left(\e\right)$ can be lower bounded as
\begin{equation}\label{814}
 V_{nn1}\left(\e\right)\geq (\mu_m\|\e_2\|^2-2\sqrt{2}d_{1}\mu_Mx_{e}\|\e_2\|+2d_{1}\dLm x_{e}^2)/2,
\end{equation}
where $x_{e}:=\left(\sum_{i=1}^{n}\ln(\cosh(e_{1i}))\right)^{1/2}$. This comes from the fact that $\tanh^2(z)\leq 2\ln\big(\cosh(z)\big)$ provides the inequality  $\|\tanh(\e_1)\|\leq \sqrt{2}x_{e}$. Consequently, it follows that $[\tanh(\e_1)]^{\top}\M(\q)\e_2\leq \mu_M\|\tanh(\e_1)\|\|\e_2\|\leq \sqrt{2}\mu_Mx_{e}\|\e_2\|$. The right-hand side of (\ref{814}) is a quadratic form in $\|\e_2\|$ and $x_{e}$, and it is positive semidefinite provided that condition \eqref{530} holds.

Taking into account the facts that $\dot \M-2\C$ is a skew-symmetric matrix and that $\dot \M=\C^{\top}+\C$, the time derivative of $V_{A1}$ along the trajectories of \eqref{476}
 yields
\[
\dot V_{A1}(\e,\tbtta_U) =  \dot V_1+{m_2\over 2r_2} \gamma_1^{-{r_1\over 2r_2}}V_1^{r_1\over 2r_2}\dot V_1
\] 
where $\dot V_1(\e,\tbtta_U)$ satisfies (\ref{588}). 

Additionally, note that
\begin{align}
	-\gamma_1^{-{r_1\over 2r_2}}V_1^{r_1\over 2r_2}\dot V_1=&-\gamma_1^{-{r_1\over 2r_2}}\Big[\gamma_2V_0+\gamma_1{r_1\over 2r_2}\e_1^{\top}\P\lceil \e_1\rfloor^{a}+\frac{1}{2}\tbtta_U^{\top}\bGma^{-1}\tbtta_U
	\Big]^{r_1\over 2r_2}\dot V_1\nonumber\\
	\geq &-(\gamma_{1}+\gamma_{2})\gamma_1^{-{r_1\over 2r_2}}\Big[\gamma_1{r_1\over 2r_2}\e_1^{\top}\P\lceil \e_1\rfloor^{a}+\frac{1}{2}\tbtta_U^{\top}\bGma^{-1}\tbtta_U\Big]^{r_1\over 2r_2}\Big[{1\over 4}v_1(\e)+{1\over 4}v_1(\e)\Big]\nonumber\\
	\geq & \Big[{r_1\over 2r_2}\e_1^{\top}\P\lceil \e_1\rfloor^{a}\Big]^{r_1\over 2r_2}{\gamma_{1}d_{1}\Pm\over 4}\sum_{i=1}^{n}\tanh(e_{1i}) \lceil e_{1i}\rfloor^{a}\nonumber\\
	&+(\gamma_{1}+\gamma_{2})\gamma_1^{-{r_1\over 2r_2}}\Big[\frac{1}{2}\tbtta_U^{\top}\bGma^{-1}\tbtta_U\Big]^{r_1\over 2r_2}{\dLm\over 4}\|\e_2\|^2,\nonumber\\
	\geq &  \gamma_{1}\eta_1\|\e_1\|\sum_{i=1}^{n}\tanh(e_{1i}) \lceil e_{1i}\rfloor^{a}+(\gamma_{1}+\gamma_{2})\eta_2\|\tbtta_U\|^{r_1\over r_2}\|\e_2\|^2,
\end{align}
where 
\[
\eta_1:=\left[{r_1\over 2r_2}\Pm\right]^{r_1\over 2r_2}{d_1\Pm\over 4}, \quad \eta_2:=\gamma_1^{-{r_1\over 2r_2}}\left[\frac{1}{2}\lambda_m\{\bGma^{-1}\}\right]^{r_1\over 2r_2}{\dLm\over 4}.
\]
Under this condition we can ensure that $\dot{V}_{A1}\leq0$. Moreover, it satisfies 
\begin{align*}
\dot V_{A1}(\e,\tbtta_U)\leq& -(\gamma_1+\gamma_2)\left[{\dm\over 2}\|\e_2\|^{1+b}+{\dLm\over 2}\|\e_2\|^2+\lambda_m\{\bUps_{U}\}\zeta_{1}\left(\Delta\right) \|\tbtta_U\|^{1+b}\right]\\
&-(\gamma_1+\gamma_2)\eta_2\|\tbtta_U\|^{r_1\over r_2}\|\e_2\|^2\\
&-(\gamma_1+\gamma_2)\gamma_{A1}\left[\eta_1\|\e_1\|+{1\over 2}{ d_1\Pm}\right]\sum_{i=1}^{n}\tanh(e_{1i}) \lceil e_{1i}\rfloor^{a}.
\end{align*}
Now, we focus on the latter term. Define $\eta_3:= \min\{\eta_1,d_{1}\Pm/2\}$ and note that there exists a constant $k_{t}>0$ such that 
\[
\sum_{i=1}^{n}\tanh(e_{1i}) \lceil e_{1i}\rfloor^{a}\geq k_{t}\sum_{i=1}^{n}\sat(e_{1i}) \lceil e_{1i}\rfloor^{a}\geq \begin{cases}
     k_{t}\sum_{i=1}^{n} |e_{1i}|^{1+a}& \hbox{if} \quad  \|\e_1\|<1,\\
    k_{t}\sum_{i=1}^{n} |e_{1i}|^{a} & \hbox{if} \quad \|\e_1\|\geq 1.
 \end{cases}
 \]
Using Lemma \ref{lemaMeanP} to lower-bound the last term in the inequality above, we obtain
\[
\gamma_{A1}\left[\eta_1\|\e_1\|+{1\over 2}{ d_{1}\Pm}\right]\sum_{i=1}^{n}\tanh(e_{1i}) \lceil e_{1i}\rfloor^{a}\geq \gamma_{A1}\eta_3k_{t}\|\e_1\|^{1+a},
\] 
for all $\e_1\in\mathbb{R}^{n}$. Taking this into account, we arrive at (\ref{TDofVA1}).

\subsection{On the function $V_{A2}$: proof of Claim \ref{VA2negdef} }
\label{app:VA2}

The time derivative of $V_{A2}(\e,\tbtta_U) := v_{A2}\big(V_{A1}(\e,\tbtta_U), V_{e}(\e,\tbtta_U)\big)$ with $V_{A1}$ and $V_{e}$ as in \eqref{SLFV1} and \eqref{660a} respectively, yields
\begin{align}
	\dot{V}_{A2}(\e,\tbtta_U) =&\left(\gamma_3+\gamma_{4}\right)\left[1+\frac{m_3}{2r_2}V_{A1}^{\frac{m_3-2r_2}{2r_2}}\right]\dot{V}_{A1}\nonumber\\
&-\gamma_3d_{2}\left[{\rm{d\over dt}}{z^{m_3-r_2\over m_c}{\z^{\top}\over 1+\|\z\|}}\right]\M(\q)\e_2,\nonumber\\
  &-\gamma_3d_{2}{z^{m_3-r_2\over m_c}{\z^{\top}\over 1+ \|\z\|}}\M(\q)\dot\e_2-\gamma_3d_{2}{z^{m_3-r_2\over m_c}{\z^{\top}\over 1+\|\z\|}}\dot\M(\q)\e_2.\label{eq:DVA2}
\end{align}
Now, we proceed to obtain bounds for each term of (\ref{eq:DVA2}). We begin with
\begin{align*}
{\rm{d\over dt}}{z^{m_3-r_2\over m_c}{\z\over 1+\|\z\|}} & = \left[{\rm{d\over dt}}\tanh^{m_3-r_2\over m_c}(\|\z\|)\right]{\z\over 1+ \|\z\|}\\
&\quad +{z^{m_3-r_2\over m_c}{\rm{d\over dt}}{\z\over 1+\|\z\|}}\ =f_1(\q,\e_2,\tbtta_U),
\end{align*}
where
\begin{align*}
f_1(\q,\e_2,\tbtta_U)&:={m_3-r_2\over m_c}z^{m_3-r_2-m_c\over m_c}{\rm sech}(\|\z\|)^2{\z^{\top}\over \|\z\|}{\z\over 1+\|\z\|} \dot{\z}\\
&\quad +z^{m_3-r_2\over m_c}\left(
{\dot{\z}\over 1+\|\z\|}-{\z^{\top}\z \dot{\z}\over (1+ \|\z\|)^2\|\z\|} \right)
\end{align*}
and $\dot{\z}=\dot\bPsi(\q)  \tbtta_U+\bPsi(\q)\dot\tbtta_U$. Moreover, we have
\begin{align*}
\left\|{\z^{\top}\over \|\z\|}\left(\dot\bPsi(\q)  \tbtta_U+\bPsi(\q)\dot\tbtta_U\right){\z\over 1+\|\z\|}\right\|&\leq z\|\dot\bPsi(\q)  \tbtta_U+\bPsi(\q)\dot\tbtta_U\|\\
 & \leq z\left[\|\dot\bPsi(\q)  \tbtta_U\| +\|\bPsi(\q)\dot\tbtta_U\|\right],\\
\left\|{\dot{\z}\over 1+\|\z\|}-{\z^{\top}\dot{\z}\z\over (1+ \|\z\|)^2\|\z\|} \right\|&
\leq \left[1+{1\over (1+\|\z\|)}\right]\|\dot{\z}\|\leq 2\dot{\z},
\end{align*}
where we used the triangle inequality, the properties of $\tanh(\cdot)$, and the inequality $1\leq 1+\|\z\|$. Taking into account that ${\rm sech}^2(\|\z\|)\leq1$, the function $f_1(\cdot)$  can be bounded as
\begin{align*}
\left\|{\rm{d\over dt}} {\tanh(\|\z\|)^{m_3-r_2\over m_c}  {\z\over 1+\|\z\|}} \right\|&
\leq k_{M}z^{m_3-r_2\over m_c}\left[\|\dot\bPsi(\q)  \tbtta_U\|+\|\bPsi(\q)\dot\tbtta_U\|\right], 
\end{align*}
where $k_{M}:={m_3-r_2\over m_c}+2$. 

Recalling that $\|\dot \bPsi(\q)\|\leq L_{\Psi}\|\e_2\|$, we proceed to bound the $f_1^{\top}(\q,\e_2,\tbtta_U)\M(\q)\e_2$ as follows 
\begin{align*}
f_1^{\top}(\q,\e_2,\tbtta_U)\M(\q)\e_2&\leq \mu_M k_{M}z^{m_3-r_2\over m_c}\left[\|\dot\bPsi(\q)  \tbtta_U\|+\|\bPsi(\q)\dot\tbtta_U\|\right]\|\e_2\| \\
&\leq \mu_M k_{M}L_{\Psi}z^{m_3-r_2\over m_c}\|\tbtta_U\|\|\e_2\|^2,\\
&\quad +\mu_M k_{M}z^{m_3-r_2\over m_c}\|\bPsi(\q)\dot\tbtta_U\|\|\e_2\|,
\end{align*}
where the latter term corresponds to 
$$
z^{m_3-r_2\over m_c}\|\bPsi(\q)\dot\tbtta_U\|\|\e_2\|=
z^{m_3-r_2\over m_c}\big\|-\bPsi(\q)\bGma\bPsi^{\top}(\e_1+\q_{d})
\left[\gamma_1d_1\tanh(\e_1)+\left(\gamma_1+\gamma_2\right)\e_2\right]-(\gamma_1+\gamma_2)\bPsi(\q)\bGma\bUps_{U}\zeta_{1}(\Delta)\lceil \tbtta_U\rfloor ^{b}\big\|\|\e_2\|$$
and admits the following bound
\begin{align*}
z^{m_3-r_2\over m_c}\|\bPsi(\q)\dot\tbtta_U\|\|\e_2\|
 &\leq \gamma_1d_1\lambda_{M}\{\bGma\}z^{m_3-r_2\over m_c}\|\bPsi(\q)\|^2\|\tanh(\e_1)\|\|\e_2\|\\
 +&z^{m_3-r_2\over m_c}\left(\gamma_1+\gamma_2\right)\lambda_{M}\{\bGma\}\|\bPsi(\q)\|^2\|\e_2\|^2\\
 + & \left(\gamma_1+\gamma_2\right)\lambda_{M}\{\bGma\}\lambda_{M}\{\bUps_{U}\}z^{m_3-r_2\over m_c}\|\bPsi(\q)\|\zeta_{1}\left(\Delta\right)\|\lceil \tbtta_U\rfloor ^{b}\|\|\e_2\|.
\end{align*}
Now, using the fact that $\dot\M=\C^{\top}+\C$, we get that
\begin{align*}
-\gamma_3d_{2}z^{m_3-r_2\over m_c}&{\z^{\top}\over 1+\|\z\|}\left(\M(\q)\dot\e_2
-\dot\M(\q)\e_2\right) =\\
=&\gamma_3d_{2}{z^{m_3-r_2\over m_c}{\z^{\top}\over 1+\|\z\|}}\left[\P\left\lceil \e_1\right\rfloor ^{a}+\D\left\lceil \e_2\right\rfloor ^{b}+\right.\left. \D_L\e_2-\C^{\top}(\q,\e_2)\e_2\right]\\
& -\gamma_3d_{2}z^{m_3-r_2\over m_c}{\|\z\|^2\over 1+\|\z\|}.
\end{align*}

Recall that $z:=\tanh(\|\z\|)$. From Lemmas \ref{lemaMeanP}-\ref{lemaMeanP0} and \ref{Lemboundtnh}, we have that
\begin{align*}
z^{m_3-r_2\over m_c}{\z^{\top}\over 1+\|\z\|}\P\left\lceil \e_1\right\rfloor ^{a}\leq & A_1z^{m_3-r_2+m_c\over m_c}\|\e_1\|^{a},\\
z^{m_3-r_2\over m_c}{\z^{\top}\over 1+\|\z\|}\D\left\lceil \e_2\right\rfloor ^{b}\leq & A_0z^{m_3-r_2+m_c\over m_c}\|\e_2\|^{b},\\
z^{m_3-r_2\over m_c}{\z^{\top}\over 1+\|\z\|}\D_L\e_2\leq& d_{LM}z^{m_3-r_2+m_c\over m_c}\|\e_2\|,\\
-z^{m_3-r_2\over m_c}{\z^{\top}\over 1+\|\z\|}\C^{\top}(\q,\e_2)\e_2\leq& L_cz^{m_3-r_2+m_c\over m_c}\|\e_2\|^2,\\
{1\over K_{t}}z^{m_3+3r_2-2r_1\over m_c}\leq{1\over K_{t}}z^{m_3-r_2+m_c\over m_c}\|\z\|\leq &z^{m_3-r_2\over m_c}{\|\z\|^2\over 1+\|\z\|}.
\end{align*}
where, from Lemma \ref{lemaMeanP0}, $A_1=n^{(1-a)/2}p_{M}$, $A_0=n^{(1-b)/2}\dM$. 
Moreover, 
\begin{align*}
\left(1+\frac{m_2}{2r_2}V_{A1}^{\frac{m_3-2r_2}{2r_2}}\right)\dot{V}_{A1} & \leq-\left[1+\frac{m_3}{2r_2}\left(\gamma_2V_0+\gamma_1{r_1\over 2r_2}\e_1^{\top}\P\lceil \e_1\rfloor^{a}+\frac{1}{2}\tbtta_U^{\top}\bGma^{-1}\tbtta_U\right)^{\frac{m_3-2r_2}{2r_2}}\right]\times\\
 & \quad \left(\gamma_1+\gamma_2\right)\left[\gamma_{A1}c_1\bar{v}_1\left(\e\right)+{\dm\over 2}\|\e_2\|^{1+b}+\lambda_m\{\bUps_{U}\}\zeta_{1}\left(\Delta\right)\|\tbtta_U\|^{1+b}+\eta_{2}\|\tbtta_U\|^{r_1\over r_2}\|\e_2\|^2\right]\\
 & \leq-\left(\gamma_1+\gamma_2\right)\left[1+\frac{m_3}{2r_2}\left(\gamma_2V_0\right)^{\frac{m_3-2r_2}{2r_2}}\right]\times\\
 & \quad \left[\gamma_{A1}c_1\bar{v}_1\left(\e\right)+{\dm\over 2}\|\e_2\|^{1+b}+\lambda_m\{\bUps_{U}\}\zeta_{1}\left(\Delta\right)\|\tbtta_U\|^{1+b}+\eta_{2}\|\tbtta_U\|^{r_1\over r_2}\|\e_2\|^2\right].
\end{align*}

Substituting all the above relations into \eqref{eq:DVA2}, we obtain
\begin{align*}
\dot{V}_{A2}(\e,\tbtta_U)  \leq &-\left(\gamma_3+\gamma_{4}\right)\left(\gamma_1+\gamma_2\right)\left[1+\frac{m_3}{2r_2}\left(\gamma_2V_0\right)^{\frac{m_3-2r_2}{2r_2}}\right]\left[\gamma_{A1}c_1\bar{v}_1\left(\e\right)+\right.\\
&\left.+{\dm\over 2}\|\e_2\|^{1+b}+\lambda_{M}\{\bUps_{U}\}\zeta_{1}\left(\Delta\right)\|\tbtta_U\|^{1+b}+\eta_{2}\|\tbtta_U\|^{r_1\over r_2}\|\e_2\|^2\right] \\
 & +\gamma_3d_{2}k_{M}\mu_M\gamma_1d_1\lambda_{M}\{\bGma\}z^{m_3-r_2\over m_c}\|\bPsi(\q)\|^2\|\tanh(\e_1)\|\|\e_2\|\\
 &+\gamma_3d_{2}k_{M}\mu_Mz^{m_3-r_2\over m_c}\left[\left(\gamma_1+\gamma_2\right)\lambda_{M}\{\bGma\}\|\bPsi(\q)\|^2+L_{\Psi}\|\tbtta_U\|\right]\|\e_2\|^2	\\
 &+\gamma_3d_{2}k_{M}\left(\gamma_1+\gamma_2\right)\lambda_{M}\{\bGma\}\lambda_{M}\{\bUps_{U}\}\mu_Mz^{m_3-r_2\over m_c}\|\bPsi(\q)\|\zeta_{1}\left(\Delta\right)\|\lceil \tbtta_U\rfloor ^{b}\|\|\e_2\|\\
 &
 + \gamma_3d_{2}z^{m_3-r_2+m_c\over m_c}\left[A_1\|\e_1\|^{a}+A_0\| \e_2\|^{b}+d_{LM}\|\e_2\|\right]-\gamma_3{d_{2}\over K_{t}}z^{\frac{4r_2}{m_c}}.
\end{align*}
Recalling that $\|\bPsi(\q)\|\leq k_{\Psi}$,
$z\leq1$, $2r_2>r_1> r_2> m_c$,
$z\leq\|\z\|$ and $\gamma_{A2}=\gamma_3/(\gamma_3+\gamma_{4})$, we get
\begin{align*}
\frac{\dot{V}_{A2}(\e,\tbtta_U)}{\gamma_3+\gamma_{4}}  \leq &-\left(\gamma_1+\gamma_2\right)\gamma_{A1}c_1\left[1+\frac{m_3}{2r_2}\left(\gamma_2V_0\right)^{\frac{m_3-2r_2}{2r_2}}\right]\bar{v}_1\left(\e\right)-{\gamma_{A2}d_{2}\over K_{t}}z^{\frac{4r_2}{m_c}}\\
 &+ \gamma_{A2}d_{2}k_{M}\mu_M\lambda_{M}\{\bGma\}\left|z\right|^{\frac{m_2-r_2}{m_c}}\left\{ \gamma_1d_1\|\bPsi(\q)\|^2\|\tanh(\e_1)\|\|\e_2\|+\left[\left(\gamma_1+\gamma_2\right)\|\bPsi(\q)\|^2\right.\right.\\
 &\left.\left.+L_{\Psi}\lambda_{M}^{-1}\{\bGma\}\|\tbtta_U\|\right]\|\e_2\|^2\right\}+ \gamma_{A2}d_{2}z^{\frac{m_2-r_2+m_c}{m_c}}\left[A_1\|\e_1\|^{a}+A_0\|\e_2\|^{b}+d_{LM}\|\e_2\|\right]\\
 & -\left(\gamma_1+\gamma_2\right)\left[1+\frac{m_3}{2r_2}\left(\gamma_2V_0\right)^{\frac{m_3-2r_2}{2r_2}}\right]\left[{\dm\over2}\|\e_2\|^{1+b}+\lambda_m\{\bUps_{U}\}\zeta_{1}\left(\Delta\right)\|\tbtta_U\|^{1+b}\right]\\
 &+ \gamma_{A2}d_{2}k_{M}\left(\gamma_1+\gamma_2\right)\lambda_{M}\{\bGma\}\lambda_{M}\{\bUps_{U}\}\mu_M\|\bPsi(\q)\|\zeta_{1}\left(\Delta\right)z^{m_2-r_2\over m_c}\| \tbtta_U\|^{b}\|\e_2\|\\
 &-\left(\gamma_1+\gamma_2\right)\eta_2\|\tbtta_U\|^{r_1\over r_2}\|\e_2\|^2.
\end{align*}
Considering that $\bar{v}_1\left(\e\right)=\|\e_1\|^{\frac{2r_2}{r_1}}+\|\e_2\|^2$ and
$V_0\geq\frac{\mu_m}{2}\|\e_2\|^2+\frac{r_1}{2r_2}\Pm\|\e_1\|^{\frac{2r_2}{r_1}}$,
the first term can be lower-bounded as
\begin{align*}
&\left[1+\frac{m_3}{2r_2}\left(\gamma_2V_0\right)^{\frac{m_3-2r_2}{2r_2}}\right]\bar{v}_1\left(\e\right) 
  \geq \|\e_1\|^{\frac{2r_2}{r_1}}+\|\e_2\|^2	\\
 &\qquad\qquad+\frac{m_3}{2r_2}\left(\frac{\gamma_2r_1\Pm}{2r_2}\right)^{\frac{m_3-2r_2}{2r_2}}\|\e_1\|^{\frac{m_3}{r_1}}+\frac{m_3}{2r_2}\left(\frac{\gamma_2\mu_m}{2}\right)^{\frac{m_3-2r_2}{2r_2}}\|\e_2\|^{\frac{m_3}{r_2}}.
\end{align*}	

Using the facts that for any $\alpha,\beta>0$ such that $\left|z\right|^{\alpha}\leq k_{\Psi}^{\alpha-\beta}\left|z\right|^{\beta}$, for $\|\bPsi(\q)\|\leq k_{\Psi}$, $|z|\leq1$, $m_3-r_2=2r_1$, we obtain that
\begin{align*}
	z^{\frac{m_3-r_2}{m_c}}\|\tanh(\e_1)\|\|\e_2\|= & z^{\frac{m_3(m_3-r_2)-4r_1r_2}{m_3m_c}}z^{\frac{4r_1r_2}{m_cm_3}}\|\e_1\|\|\e_2\|\leq z^{\frac{4r_1r_2}{m_cm_2}}\|\e_1\|\|\e_2\|,\\
	& \quad  \forall m_3(m_3-r_2)\geq 4r_1r_2\iff 2r_1\geq r_2\implies r_1\geq r_2,\\
z^{\frac{m_3-r_2}{m_c}}\|\e_2\|^2= &	z^{\frac{m_3-r_2}{m_c}-{r_1\over r_2}}z^{r_1\over r_2}\|\e_2\|^2\leq k_{\Psi}^{r_1\over r_2}\|\tbtta_U\|^{r_1\over r_2}\|\e_2\|^2,\\
	&\quad \forall (m_2-r_2)r_2\geq m_cr_1\iff  2r_1r_2\geq (2r_2-r_1)r_1\implies r_1\geq 0,\\
z^{\frac{m_3-r_2}{m_c}}\|\tbtta_U\|\|\e_2\|^2=	&\left|z\right|^{\frac{m_3-r_2}{m_c}-{r_1-r_2\over r_2}}\left|z\right|^{{r_1-r_2\over r_2}}\|\tbtta_U\|\|\e_2\|^2\leq  k_{\Psi}^{r_1-r_2\over r_2}\|\tbtta_U\|^{r_1\over r_2}\|\e_2\|^2,\\
	& \quad \forall (m_3-r_2)r_2\geq m_c(r_1-r_2)\iff r_1^2+2r_2^2\geq r_1r_2\implies 2r_2\geq r_1\geq r_2,\\
z^{\frac{m_3-r_2+m_c}{m_c}}\|\e_2\| =& z^{\frac{m_3-r_2+m_c}{m_c}-{2r_2\over m_c}}z^{\frac{2r_2}{m_c}}\|\e_2\|\leq z^{\frac{2r_2}{m_c}}\|\e_2\|,\\
& \quad \forall m_3-r_2+m_c\geq 2r_2\iff r_1\geq  0,\\
z^{\frac{m_3-r_2+m_c}{m_c}}=& z^{\frac{m_3-r_2+m_c}{m_c}-{2r_1\over m_c}}z^{2r_1\over m_c}\leq z^{2r_1\over m_c}, \quad \forall m_3-r_2+m_c\geq 2r_1\iff 2r_2\geq  r_1.
\end{align*}

And thus, $\dot{V}_{A2}(\e,\tbtta_U)$ can be written as
\[
\frac{\dot{V}_{A2}(\e,\tbtta_U)}{\gamma_{A2}d_{2}\left(\gamma_3+\gamma_{4}\right)} \leq -W_1\left(\e,z\right)-W_2\left(\e,z\right)-\left(\gamma_1+\gamma_2\right)W_3(\e,\tbtta_U)-W_{4}(\e,\tbtta_U),
\]
where
\begin{align*}
W_1\left(\e,z\right)  =& \frac{\left(\gamma_1+\gamma_2\right)\gamma_{A1}c_1}{\gamma_{A2}d_{2}}\frac{m_3}{2r_2}\left[\left(\frac{\gamma_2r_1\Pm}{2r_2}\right)^{\frac{m_3-2r_2}{2r_2}}\|\e_1\|^{\frac{m_3}{r_1}}+\left(\frac{\gamma_2\mu_m}{2}\right)^{\frac{m_3-2r_2}{2r_2}}\|\e_2\|^{\frac{m_3}{r_2}}\right]\\
 &+\frac{1}{2K_{t}}z^{\frac{4r_2}{m_c}}-\gamma_1d_1k_{M}k_{\Psi}^2\lambda_{M}\{\bGma\}\mu_Mz^{\frac{4r_1r_2}{m_cm_3}}\|\e_1\|\|\e_2\|, \\
W_2\left(\e,z\right) =&\frac{\left(\gamma_1+\gamma_2\right)\gamma_{A1}c_1}{2\gamma_{A2}d_{2}}\left(\|\e_1\|^{\frac{2r_2}{r_1}}+\|\e_2\|^2\right)+\frac{1}{2K_{t}}z^{\frac{4r_2}{m_c}}
-d_{LM} z^{\frac{2r_2}{ m_c}}\|\e_2\|\\
&-z^{\frac{4r_2}{m_c}}\left[A_1\|\e_1\|^{a}+A_0\|\e_2\|^{b}\right],\\
W_3(\e,\tbtta_U)  =&\frac{1}{\gamma_{A2}d_{2}}\left[1+\frac{m_3}{2r_2}\left(\gamma_2V_0\right)^{\frac{m_3-2r_2}{2r_2}}\right]\left({\dm\over 2}\|\e_2\|^{1+b}+\lambda_m\{\bUps_{U}\}\zeta_{1}\left(\Delta\right)\|\tbtta_U\|^{1+b}\right)\\
& -k_{M}\lambda_{M}\{\bGma\}\lambda_{M}\{\bUps_{U}\}\mu_Mk_{\Psi}\zeta_{1}\left(\Delta\right)\| \tbtta_U\|^{b}\|\e_2\|,\\
W_{4}(\e,\tbtta_U)  =& \left({(\gamma_1+\gamma_2)\eta_2\over \gamma_{A2}d_{2}}-(k_{\Psi}^{r_1-r_2\over r_2}L_{\Psi}+\left(\gamma_1+\gamma_2\right)k_{\Psi}^2k_{\Psi}^{r_1\over r_2})k_{M}\mu_M\lambda_{M}\{\bGma\}\right)\|\tbtta_U\|^{r_1\over r_2}\|\e_2\|^2.
\end{align*}

To prove that $\dot{V}_{A2}(\e,\tbtta_U)$ can be rendered negative semidefinite, it will be shown that $W_1(\e,\tbtta_U) \geq 0,W_2(\e,\tbtta_U) \geq 0$, $W_3(\e,\tbtta_U)\geq0$ and $W_{4}(\e,\tbtta_U)\geq0$ by selecting $\gamma_{A2}>0$ sufficiently small. 

First, note that if $\gamma_{A2}>0$ is chosen sufficiently small, then $W_{4}(\e,\tbtta_U)\geq0$. 

Next, with $\mathbf{r}=\left[r_1,r_2,m_cm_3/4r_2\right]$, observe that $W_1$ is $\left(\mathbf{r},m_2\right)$-homogeneous with respect to the coordinates $x=\|\e_1\|$, $y=\|\e_2\|$ and $z$,
while $W_2$ is $\left(\mathbf{r},2r_2\right)$-homogeneous, with $\mathbf{r}=\left[r_1,r_2,m_c/2\right]$, with respect to the same coordinates. The first terms of both functions are nonnegative and vanish only when $\e={0}_{2n}$. Moreover, for $\e={0}_{2n}$, we have $W_1\left({ 0}_{2n},z\right)=\bar{W}_2\left({ 0}_{2n},z\right)=\left|z\right|^{4r_2/m_c}/2K_{t}>0$
for all $z\neq0$. According to Lemma \ref{lemAndrieuH}, the functions  $W_1\left(\e,z\right)$ and $W_2\left(\e,z\right)$ can be made positive definite with respect to $\left(\e,z\right)$ by selecting $\gamma_{A2}>0$ sufficiently small. Consequently, for sufficiently small $\gamma_{A2}>0$, it follows that $W_1 \geq 0$ and $W_2 \geq 0$.

For $W_3(\e,\tbtta_U)$, consider that (since $c=b$)
$
0\leq\zeta_{1}\left(\Delta\right)=f\left(\Delta\right)\left\lceil \Delta\right\rfloor ^{b}=\frac{\left\lceil \Delta\right\rfloor ^{d}\left\lceil \Delta\right\rfloor ^{b}}{1+\left|\Delta\right|^{c+d}}\leq1
$
and if $c=b$, it follows that $
\zeta_{1}\left(\Delta\right)=\zeta_{1}^{\frac{b}{1+b}}\left(\Delta\right)\zeta_{1}^{\frac{1}{1+b}}\left(\Delta\right)\leq\zeta_{1}^{\frac{b}{1+b}}\left(\Delta\right)$. Using this fact, and defining $y:=\|\e_1\|$ and $w:=\zeta_{1}^{\frac{1}{1+b}}\left(\Delta\right)\|\tbtta_U\|$,
we obtain that $W_3\left(e,\tbtta_U\right)\geq\bar{W}_3\left(y,w\right)$,
where
\[
\begin{array}{rl}
\bar{W}_3\left(y,w\right) & =\frac{1}{\gamma_{A2}d_{2}}\left({\dm\over 2}y^{1+b}+\lambda_m\{\bUps_{U}\}w^{1+b}\right)-k_{M}\lambda_{M}\{\bGma\}\lambda_{M}\{\bUps_{U}\}\mu_Mk_{\Psi}w^{b}y.\end{array}
\]
Note that $\bar{W}_3\left(y,w\right)$ is $\left(\mathbf{r},1+b\right)$-homogeneous
with $\mathbf{r}=\left[1,1\right]$. By the same reasoning applied to
$W_1$  and $W_2$, we conclude that $\bar{W}_3\left(y,w\right)$
can be made positive definite as a function of $\left(y,w\right)$
by selecting $\gamma_{A2}>0$ sufficiently small. Consequently, $W_3(\e,\tbtta_U) \geq 0$.

Therefore, recalling that $z:=\tanh\left(\|\bPsi\left(\q\right)\tbtta_U\|\right)$, we have 
\[
\frac{\dot{V}_{A2}(\e,\tbtta_U)}{\gamma_{A2}d_{2}\left(\gamma_3+\gamma_{4}\right)}  \leq-{W}_1\left(\e,z\right)-{W}_2\left(\e,z\right).
\]
Hence,  by selecting
$\gamma_{A2}>0$ sufficiently small, $\dot{V}_{A2}(\e,\tbtta_U)$ can be upper-bounded
by a negative definite function of $\left(\e,z\right)$. This implies that $\dot{V}_{A2}(\e,\tbtta_U)$
is negative semidefinite as a function of $(\e,\tbtta_U)$,
since $z$ may be zero even when $\tbtta_U\neq{0}_{\jmath}$. Moreover, due to the homogeneity of $W_1$ and $W_2$, there exist
constants $\bar{c}_1,\bar{c}_2>0$ such that 
\[
-\bar{W}_1-\bar{W}_2\leq-\bar{c}_1\left(\|\e_1\|^{\frac{m_3}{r_1}}+\|\e_2\|^{\frac{m_3}{r_2}}+z^{\frac{4r_2}{m_c}}\right)-\bar{c}_2\left(\|\e_1\|^{\frac{2r_2}{r_1}}+\|\e_2\|^2+z^{\frac{4r_2}{m_c}}\right).
\]
Finally, since $\frac{2r_2}{r_1}<\frac{3r_2}{r_1}\leq \frac{m_3}{r_1}$
and $2<3\leq \frac{m_3}{r_2}$, we can apply Lemma \ref{LemaBilimit}
to the terms $\|\e_1\|$ and $\|\e_2\|$ in the previous expression. As a result, there
exists a constant $c_2>0$ such that $-W_1-W_2\leq-c_2\left(\|\e_1\|^{3r_2/r_1}+\|\e_2\|^{3}+z^{4r_2/m_c}\right)$. Hence, after using Lemmas \ref{lemaMeanP}-\ref{lemaMeanP0}, \eqref{ap:163} is satisfied.

\subsection{Proof of Claim \ref{LemVft}}
\label{app:LemVft}
The function $V_{FT}$ can be written as 
\begin{align*}
V_{FT}(\e,\tbtta_U) & =\gamma_{4}V_{A1}^{\frac{m_3}{2r_2}}+(\gamma_3+\gamma_{4})V_{A1}^2+\gamma_3V_{nn}(\e,\tbtta_U),\\
V_{nn2}(\e,\tbtta_U) & =V_{A1}^{\frac{m_3}{2r_2}}+d_{3}\lceil \e_1^{\top}\rfloor^{\frac{m_3-r_2}{r_1}}\M(\q)\e_2,
\end{align*}
where $V_{A1}$ is defined as in \eqref{SLFVA1}. Moreover, since
\begin{align*}
V_{nn2}(\e,\tbtta_U) & \geq\left[\gamma_2V_0+\gamma_1V_{nn1}\right]^{\frac{m_3}{2r_2}}+d_{3}\lceil \e_1^{\top}\rfloor^{\frac{m_3-r_2}{r_1}}\M(\q)\e_2,
\end{align*}
with $V_{nn1}(\e)$ given by (\ref{549}), which is positive semidefinite if $\dLm\mu_m/\mu_M^2\geq d_{1}$. From this, we have that $V_{nn2}(\e,\tbtta_U)\geq\gamma_2^{\frac{m_3}{2r_2}}V_{nl}\left(\e\right)$,
where
\begin{align*}
	V_{nl}\left(\e\right)&=V_0^{\frac{m_3}{2r_2}}+d_{3}\gamma_2^{-{m_3\over 2r_2}}\lceil \e_1^{\top}\rfloor^{\frac{m_3-r_2}{r_1}}\M(\q)\e_2\nonumber\\
	&=\left({r_1\over 2r_2}\e_1^{\top}\P\lceil \e_1\rfloor^{a} + {1\over 2}\e_2^{\top}\M(\q)\e_2\right)^{\frac{3}{2}}+d_{3}\gamma_2^{-{m_3\over 2r_2}}\lceil \e_1^{\top}\rfloor^{\frac{m_3-r_2}{r_1}}\M(\q)\e_2\nonumber.
\end{align*}

Invoking (\ref{IneqLem}) of Lemma \ref{LemaVnn2} with $\x=\e_1$, $\y=\e_2$, $\z=\q$, $d_0=d_{3}$, $r=r_1$, $s=r_2$,  and $m=m_3$, we have that if $d_{3}$ satisfies (\ref{d4value}), then $V_{nn}\geq 0$, 
we conclude that when (\ref{d4value}) holds, then $V_{nl}\geq0$. Therefore, if $d_{3}$ satisfies (\ref{d4value}), it follows that  $V_{FT}>0$.

The time derivative of $V_{FT}$ along (\ref{476}) is given by
\begin{align*}
\frac{1}{\gamma_3+\gamma_{4}}\dot V_{FT}  = & \left(\frac{3}{2}V_{A1}^{\frac{1}{2}}+2V_{A1}\right)\dot{V}_{A1}+\gamma_{A2}d_{3}\frac{2r_2}{r_1}[\e_2|\e_1|^{a}]^{\top}\M(\q)\e_2 -\gamma_{A2}d_{3}\lceil \e_1^{\top}\rfloor^{a}\C^{\top}(\q,\e_2)\e_2\\
& -\gamma_{A2}d_{3}\lceil \e_1^{\top}\rfloor^{a}\left(\P\lceil \e_1\rfloor^{a}+\D\lceil \e_2\rfloor^{b}+\D_L\e_2-\bPsi\left(\q\right)\tbtta_U\right),
\end{align*}
where $|\e_1|\lceil \e_2\rfloor^{a}:=\begin{bmatrix}
|e_{1,1}|\lceil e_{2,1}\rfloor^{a} &	  ,\dots &, |e_{1,n}|\lceil e_{2,n}\rfloor^{a}
\end{bmatrix}^{\top}$, and $\gamma_{A2}=\gamma_3/(\gamma_3+\gamma_{4})$. 

Using Lemma \ref{lemnormprodvectors}, we have that there exists a constant $c_{12}>0$ such that 
$\frac{2r_2}{r_1}[\e_2|\e_1|^{a}]^{\top}\M(\q)\e_2\leq c_{12}\|\e_1\|^{a}\|\e_2\|^2$. And from Lemmas \ref{lemaMeanP0} and \ref{lemaMeanP}, we obtain that 
$\lceil \e_1^{\top}\rfloor^{2r_2/r_1}\C^{\top}(\q,\e_2)\e_2\leq L_c\|\e_1\|^{2r_2/ r_1}\|\e_2\|^2$, $\lceil \e_1^{\top}\rfloor^{2r_2/r_1}\D\lceil \e_2\rfloor^{b}\leq A_0\|\e_1\|^{2r_2/ r_1}\|\e_2\|^{b}$, with $A_0=n^{(1-b)/2}\dM$, and $\lceil \e_1^{\top}\rfloor^{a}\bPsi\left(\q\right)\tbtta_U\leq k_{\psi}\|\e_1\|^{2r_2/ r_1}\|\tbtta\|$.

Taking into account \eqref{TDofVA1}, there exist constants $\kappa_{i}>0$,
for $1\leq i \leq n$, independent of $\gamma_{A2}$, such that
\[
\begin{array}{rl}
	- {3\over2}V_{A1}^{1\over 2}\dot V_{A1}\leq &\kappa_1\|\e_2\|^{\frac{4r_2-r_1}{r_2}}-\kappa_2\|\e_1\|^{\frac{3r_2}{r_1}}-\kappa_3\zeta_{1}(\Delta)\|\tbtta_U\|^{\frac{4r_2-r_1}{r_2}}-\kappa_{4}\|\e_2\|^{\frac{3r_2}{r_2}}-\kappa_{5}\|\e_1\|^{\frac{2r_2}{r_1}}\|\tbtta_U\| \\
-V_{A1}\dot V_{A1}\leq& 	-\kappa_{6}\|\e_1\|^{2(r_1+r_2)\over r_1}-\kappa_{7}\zeta_{1}(\Delta)\|\tbtta_U\|^{3+b}-\kappa_{8}\|\e_1\|^{2r_2\over r_1}\|\e_2\|^2.
\end{array}
 \]
From these inequalities, we then arrive at	
\begin{equation}\label{TDofV3}
\begin{array}{rl}
{ 1\over \gamma_3+\gamma_{4}}\dot V_{FT}\leq & -\gamma_{A2} [v_1(\e_1,\e_2)+v_2(\e_1,\e_2)]-\kappa_3\zeta_{1}(\Delta)\|\tbtta_U\|^{\frac{4r_2-r_1}{r_2}}-\kappa_{7}\zeta_{1}(\Delta)\|\tbtta_U\|^{\frac{5r_2-r_1}{r_2}}-\kappa_{6}\|\e_1\|^{2(r_2+r_1)\over r_1}\\
&-(\kappa_{5}-\gamma_{A2} d_{3}k_{\Psi})\|\tbtta_U\|\|\e_1\|^{2r_2\over r_1}-(\kappa_{8}-\gamma_{A2}d_{3}L_c)\|\e_1\|^{2r_2\over r_1}\|\e_2\|^2,
\end{array}
\end{equation}
where
\begin{align*}
v_1(\e_1,\e_2) & :=p_{m}d_{3}\|\e_1\|^{\frac{4r_2-r_1}{r_1}}-d_{3}c_{12}\|\e_1\|^{\frac{2r_2-r_1}{r_1}}\|\e_2\|^2+d_{M}d_{3}\| \e_1\|^{\frac{2r_2}{r_1}}\| \e_2\|^{\frac{2r_2-r_1}{r_2}}+\frac{\kappa_1}{\gamma_{A2}}\|\e_2\|^{\frac{4r_2-r_1}{r_2}}\\
v_2(\e_1,\e_2) & :=\frac{1}{\gamma_{A2}}\left(\kappa_2\|\e_1\|^{\frac{3r_2}{r_1}}+\kappa_{4}\|\e_2\|^{\frac{3r_2}{r_2}}\right)-d_{LM}d_{3}\| \e_1\|^{\frac{2r_2}{r_1}}\|\e_2\|.
\end{align*}	
	
First, note that  the latter terms in (\ref{TDofV3})   are nonnegative for $\gamma_{A2}>0$  satisfying  $\kappa_{5}/(d_{3}k_{\Psi})\geq\gamma_{A2}$ and $\kappa_{8}/d_{3}L_c\geq \gamma_{A2}$. Second, with ${\bf r}=[r_1, r_2]$, the functions $v_1$ and $v_2$ are ${\bf r}$-homogeneous of degrees $4r_2-r_1$ and $m_3$, respectively, with respect to the coordinates $x=\|\e_1\|$ and $y=\|\e_2\|$. Hence, by Lemma \ref{lemAndrieuH}, for sufficiently small $\gamma_{A2}$, $v_1$ and $v_2$ are positive definite with respect to $\e$. Therefore,  for any $\gamma_{A2}$ sufficiently small, $\dot V_{FT}\leq0$. In fact, under this condition and by Lemma  \ref{LemmaBhatV1V2}, there exists constants $c_{10},c_{20}>0$ such that 
\begin{align*}
v_1\left(\e_1,\e_2\right)\geq c_{10}\left(\|\e_1\|^{\frac{4r_2-r_1}{r_1}}+\|\e_2\|^{\frac{4r_2-r_1}{r_2}}\right),\,v_2\left(\e_1,\e_2\right)\geq c_{20}\left(\|\e_1\|^{\frac{3r_2}{r_1}}+\|\e_2\|^{\frac{3r_2}{r_2}}\right).
\end{align*}

Moreover, from  Lemma \ref{LemaBilimit}, it follows that
\begin{itemize}
\item there exists a constant $\lambda_1\leq{\overline\lambda}_1(c_{10}/2,\kappa_{6}/2\gamma_{A2},3r_2/r_1,(4r_2-r_1)/r_2,2(r_2+r_1)/r_1)$ such that $
\lambda_1\|\e_1\|^{\frac{4r_2-r_1}{r_2}}\leq{c_{10}\over2}\|\e_1\|^{\frac{4r_2-r_1}{r_1}}+{\kappa_{6}\over2\gamma_{A2}}\|\e_1\|^{\frac{2\left(r_1+r_2\right)}{r_1}}, \quad \forall r_1>r_2$;
	
 \item there exists a constant $\lambda_{4}\leq{\overline\lambda}_{4}(c_{20}/2,\kappa_{6}/2\gamma_{A2},3r_2/ r_1, 3,2(r_2+r_1)/r_1)$ such that $
 \lambda_{4}\|\e_1\|^{3}\leq{c_{10}\over2}\|\e_1\|^{\frac{4r_2-r_1}{r_1}}+{\kappa_{6}\over2\gamma_{A2}}\|\e_1\|^{\frac{2\left(r_1+r_2\right)}{r_1}}, \quad \forall 2r_2>r_1>r_2$;

 \item there exists a constant  $\lambda_{5}\leq{\overline\lambda}_{5}(\kappa_3/2,\kappa_{7},(4r_2-r_1)/ r_2, 3,(5r_2-r_1)/r_2)$  such that $ \lambda_{5}\|\tbtta_U\|^{3r_2\over r_2}\leq {\kappa_3\over 2}\|\tbtta_U\|^{4r_2-r_1\over r_2}+\kappa_{7}\|\tbtta_U\|^{5r_2-r_1\over r_2},  \quad \forall 2r_2>r_1>r_2$.
\end{itemize}
Therefore, 
\begin{align*}
	v_1\left(\e_1,\e_2\right)+&v_2\left(\e_1,\e_2\right)+\frac{\kappa_{6}}{\gamma_{A2}}\|\e_1\|^{\frac{2\left(r_1+r_2\right)}{r_1}}\geq {c_{10}\over 2}\|\e_1\|^{4r_2-r_1\over r_1}+\lambda_1\|\e_1\|^{\frac{4r_2-r_1}{r_2}}\\
	&+c_{10}\|\e_2\|^{4r_2-r_1\over r_2} +{c_{20}\over 2}\|\e_1\|^{3r_2\over r_1}+\lambda_{4}\|\e_1\|^{\frac{3r_1}{r_1}}+c_{20}\|\e_2\|^{\frac{3r_2}{r_2}}
\end{align*}
Furthermore, since $\zeta_{1}(\Delta)\geq0$, 
\[
{\kappa_3}\zeta_{1}(\Delta)\|\tbtta_U\|^{4r_2-r_1\over r_2 }+{\kappa_{7}}\zeta_{1}(\Delta)\|\tbtta_U\|^{5r_2-r_1\over r_2}\geq{\zeta_{1}(\Delta)}\left(
{\kappa_3\over 2}\|\tbtta_U\|^{4r_2-r_1\over r_2}+ {\lambda_{5}}\|\tbtta_U\|^{3r_2\over r_2}\right).
\]

Substituting these bounds into (\ref{TDofV3}), we obtain 

\begin{equation}\label{TDofV3ub}
\begin{aligned}
\frac{1}{\gamma_3+\gamma_{4}}\dot V_{FT}\leq &-\gamma_{A2}\left({c_{10}\over 2}\|\e_1\|^{4r_2-r_1\over r_1}+\lambda_1\|\e_1\|^{\frac{4r_2-r_1}{r_2}}+ {c_{20}\over 2}\|\e_1\|^{3r_2\over r_1}\right.\\
&\left.+\lambda_{4}\|\e_1\|^{\frac{3r_1}{r_1}}+c_{10}\|\e_2\|^{4r_2-r_1\over r_2} +c_{20}\|\e_2\|^{\frac{3r_2}{r_2}}\right)\\
&-\zeta_{1}(\Delta)\left(
{\kappa_3\over 2}\|\tbtta_U\|^{4r_2-r_1\over r_2}+ {\lambda_{5}}\|\tbtta_U\|^{3r_2\over r_2}\right).
\end{aligned}
\end{equation}

Since $1>\zeta_{1}(\Delta)\geq0$ holds for all $\Delta$, the right-hand side of
of (\ref{TDofV3}) can be upper-bounded as
\begin{align*}
\frac{1}{\gamma_3+\gamma_{4}}\dot V_{FT}\leq &-\gamma_{A2}\zeta_{1}(\Delta)\left({c_{10}\over 2}\|\e_1\|^{4r_2-r_1\over r_1}+\lambda_1\|\e_1\|^{\frac{4r_2-r_1}{r_2}}+ {c_{20}\over 2}\|\e_1\|^{3r_2\over r_1}\right.\\
&\left.+\lambda_{4}\|\e_1\|^{\frac{3r_1}{r_1}}+c_{10}\|\e_2\|^{4r_2-r_1\over r_2} +c_{20}\|\e_2\|^{\frac{3r_2}{r_2}}\right)\\
&-\zeta_{1}(\Delta)\left(
{\kappa_3\over 2}\|\tbtta_U\|^{4r_2-r_1\over r_2}+ {\lambda_{5}}\|\tbtta_U\|^{3r_2\over r_2}\right).
\end{align*}
or equivalently,
\begin{equation}\label{TDofV3b}
{1\over \gamma_3+\gamma_{4}}\dot V_{FT}\leq-\gamma_{A2}\zeta_{1}(\Delta) \vartheta(\e_1,\e_2,\tbtta_U)
\end{equation}
where  $\vartheta(\e_1,\e_2,\tbtta_U)=\bar v_1(\e_1,\e_2,\tbtta)+
\bar v_2(\e_1,\e_2,\tbtta_U)$,
\begin{align*}
\bar v_1(\e_1,\e_2,\tbtta):=& c_{10}\|\e_2\|^{4r_2-r_1\over r_2} + {c_{10}\over 2}\|\e_1\|^{4r_2-r_1\over r_1}+\lambda_1\|\e_1\|^{4r_2-r_1\over r_2}+{\kappa_3\over 2\gamma_{A2}}\|\tbtta_U\|^{4r_2-r_1\over r_2},\\
\bar v_2(\e_1,\e_2,\tbtta):= &c_{20}\|\e_2\|^{3r_2\over r_2}+ {c_{20}\over 2}\|\e_1\|^{3r_2\over r_1}+\lambda_{4}\|\e_1\|^{3r_2\over r_2}+{\lambda_{5}\over \gamma_{A2}}\|\tbtta_U\|^{3r_2\over r_2}.
\end{align*}
%
%
%
%
%
%
%
%
Note that $\bar v_1(\e_1,\e_2,\tbtta_U)$ and  $\bar v_2(\e_1,\e_2,\tbtta_U)$ are positive definite. Define $\epsilon=\|\e_1\|^{r_1/r_2}$, then 
$$
\bar v_1(\e_1,\e_2,\tbtta_U)=\bar v_1(\e_1,\e_2,\epsilon, \tbtta_U):=c_{10}\|\e_2\|^{4r_2-r_1\over r_2} + {c_{10}\over 2}\|\e_1\|^{4r_2-r_1\over r_1}+\lambda_1\epsilon^{4r_2-r_1\over r_1}+{\kappa_3\over 2\gamma_{A2}}\|\tbtta_U\|^{4r_2-r_1\over r_2}
$$

Defining the weights ${\bf r}=\left[r_1,r_2,r_1,r_2\right]$, observe that $\bar v_1$ is $\left({\bf r}, 4r_2-r_1\right)$-homogeneous with respect to the coordinates $x=\|\e_1\|$, $y=\|\e_2\|$, $\epsilon$ and $w=\|\tbtta_U\|$. In fact, $\bar v_1$ is the homogeneous approximation in the $0$-limit of $\vartheta$. While,
$$
\bar v_2(\e_1,\e_2,\tbtta_U)=\bar v_2(\e_1,\e_2,\epsilon,\tbtta_U):=c_{20}\|\e_2\|^{3r_2\over r_2}+ {c_{20}\over 2}\|\e_1\|^{3r_2\over r_1}+\lambda_{4}|\epsilon|^{3r_2\over r_1}+{\lambda_{5}\over \gamma_{A2}}\|\tbtta_U\|^{3r_2\over r_2}
$$
is (${\bf r},3r_2$)-homogeneous with respect to the same coordinates. In this case, $\bar v_2$ represents the homogeneous approximation in the $\infty$-limit of $\vartheta$. Therefore, $\vartheta(\e_1,\e_2,\tbtta_U)=\vartheta(\e_1,\e_2,\epsilon,\tbtta_U):=\bar v_1(\e_1,\e_2,\epsilon,\tbtta_U)+
\bar v_2(\e_1,\e_2,\epsilon,\tbtta_U)$ is a {\it bl}-homogeneous positive-definite function.

Now, we show that there exists $c_3>0$ such that 
\begin{equation}
\vartheta\left(\e_1,\e_2,\tbtta_U\right)\geq c_3\left(V_{FT}^{\frac{4r_2-r_1}{3r_2}}\left(\e_1,\e_2,\tbtta_U\right)+V_{FT}^{3\over 4}\left(\e_1,\e_2,\tbtta_U\right)\right).\label{eq:IneqThetaVa1}
\end{equation}

To show this, we introduce some functions ${}^{*}v_{A1},v_{A1,0},v_{A1,\infty},v_{A1},v_{F1},,v_{F2},,v_{FT},v_{FT,0},v_{F1,\infty}:\mathbb{R}^{n}\times\mathbb{R}^{n}\times \mathbb{R}^{1}\times\mathbb{R}^{\jmath}\rightarrow\mathbb{R}_{\geq 0}$
given by
\begin{align*}
{}^{*}v_1\left(\e_1,\e_2,\epsilon,\tbtta_U\right) &:=  \gamma_2\left(\frac{r_1}{2r_2}n^{r_1-r_2\over 2r_1}p_{M}\|\e_1\|^{\frac{2r_2}{r_1}}+\frac{\mu_M}{2}\|\e_2\|^2\right)+\gamma_1\frac{r_1}{2r_2}n^{r_1-r_2\over 2r_1}p_{M}\|\e_1\|^{\frac{2r_2}{r_1}}\\
&+2\gamma_1\left(\frac{1}{2}d_{1}d_{LM}\left|\epsilon\right|^{2r_2\over r_1}+\frac{\mu_M}{2}\|\e_2\|^2\right)+\frac{1}{2}\lambda_{M}\{\bGma^{-1}\}\|\tbtta_U\|^2,\\
v_{A1,0}\left(\e_1,\e_2,\epsilon,\tbtta_U\right)&={}^{*}v_1\left(\e_1,\e_2,\epsilon,\tbtta_U\right),\\
v_{A1,\infty}\left(\e_1,\e_2,\epsilon,\tbtta_U\right)&:= \gamma_1^{-{m_2\over 2r_2}}{}^{*}v_1^{m_2\over 2r_2}\left(\e_1,\e_2,\epsilon,\tbtta_U\right),\\
v_{A1}\left(\e_1,\e_2,\epsilon,\tbtta_U\right)&:= v_{A1,0}\left(\e_1,\e_2,\epsilon,\tbtta_U\right)+v_{A1,\infty}\left(\e_1,\e_2,\epsilon,\tbtta_U\right),\\
v_{F1}\left(\e_1,\e_2,\epsilon,\tbtta_U\right) & :=(\gamma_3+\gamma_{4})v_{A1}^{m_3\over 2r_2}(\e_1,\e_2,\epsilon,\tbtta_U)+\gamma_3d_{3}\mu_M\| \e_1\|^{\frac{m_3-r_2}{r_1}}\|\e_2\|,\\
v_{F2}\left(\e_1,\e_2,\epsilon,\tbtta_U\right) & :=(\gamma_3+\gamma_{4})v_{A1}^2(\e_1,\e_2,\epsilon,\tbtta_U),\\
v_{FT}\left(\e_1,\e_2, \epsilon,\tbtta_U\right) &:=v_{F1}\left(\e_1,\e_2,\epsilon,\tbtta_U\right) +v_{F2}\left(\e_1,\e_2,\epsilon,\tbtta_U\right),\\
v_{FT,0}\left(\e_1,\e_2,\epsilon,\tbtta_U\right) & :=(\gamma_3+\gamma_{4})v_{A1,0}^{m_3\over 2r_2}(\e_1,\e_2,\epsilon,\tbtta_U)+\gamma_3d_{3}\mu_M\| \e_1\|^{\frac{m_3-r_2}{r_1}}\|\e_2\|,\\
v_{FT,\infty}\left(\e_1,\e_2,\epsilon,\tbtta_U\right) & :=(\gamma_3+\gamma_{4})v_{A1,\infty}^2(\e_1,\e_2,\epsilon,\tbtta_U).
\end{align*}
where in ${}^{*}v_1$ we have used the inequality $\sum_{i=1}^{n}D_{li}\ln\big(\cosh(\tilde{e}_{1i})\big)\leq d_{LM}\|\e_1\|^2/2$. 

Note that  ${}^{*}v_1$ is a bound of $V_1$ while $v_{FT}$  is a bound of $V_{FT}$ , that is, $V_1\leq {}^{*}v_1$ and $V_{FT}\leq v_{FT}$. 

For the functions defined above,  we introduce the following change of coordinates $x=\|\e_1\|$, $y=\|\e_2\|$, $\epsilon$ and $w=\|\tbtta_U\|$, which has associated the weight vector ${\bf r}=\left[r_1,r_2,r_1,r_2\right]$. With respect to these coordinates, it is straightforward to verify that the function $v_{A1,0}$ (or equivalently, ${}^{\ast}v_1$) is  $({\bf r},2r_2)$-homogeneous, whereas $v_{A1,\infty}$ is $({\bf r},m_2)$-homogeneous. Consequently,  $v_{A1}$ is bl-homogeneous of degrees $m_0=2r_2$ and $m_{\infty}=m_2$.  According to Lemma \ref{AndrieuP2008CompF}, $v_{FT,\infty}$ is ${\bf r}$-homogeneous of degree $2m_2$, while $v_{FT,0}$ is $({\bf r},m_3)$-homogeneous. Moreover,   $v_{FT,0}$  is the homogeneous approximation of $v_{F1}$ in the $0$-limit and $v_{FT,\infty}$  is the homogeneous approximation of $v_{F2}$ in the $\infty$-limit. This actually implies that  $v_{FT,0}$ and $v_{FT,\infty}$ are the homogeneous approximation of $v_{FT}$ in the $0$- and $\infty$-limits. On one hand, it means that in the $0$-limit, there exists a constant $c_{3,0}\in \mathbb{R}_{>0}$ such that  ${\bar v}_1\geq c_{3,0}v_{FT,0}^{(4r_2-r_1)/ r_1}$, in accordance with Lemma \ref{LemmaBhatV1V2}. On the other hand, in the $\infty$-limit, it follows that $\bar v_2\geq c_{3,\infty}v_{FT,\infty}^{3r_2/2m_2}$,  for some $c_{3,\infty}\in \mathbb{R}_{>0}$. According to Lemma \ref{corAndrieuP2008}, it therefore
follows that there exists $c_3$ such that $\vartheta\geq c_3(v_{FT}^{(4r_2-r_1)/ m_3}+v_{FT}^{3r_2/2m_2})$. Since  $\vartheta(\e_1,\e_2,\epsilon,\tbtta_U)=\vartheta(\e_1,\e_2,\tbtta_U)$ and $v_{FT}(\e_1,\e_2,\epsilon,\tbtta_U)=v_{FT}(\e_1,\e_2,\tbtta_U)$,
then \eqref{eq:IneqThetaVa1} follows. 

Using \eqref{eq:IneqThetaVa1} in \eqref{TDofV3b}, we obtain (\ref{dotVft}).

\subsection{On the non-negativity of $v_2$ in (\ref{587})}
\label{app:onv2}

The non-negativity of $v_2(\e)$ is established by defining  $x\triangleq \|\tanh(\e_1)\|^{1+a}$, $w\triangleq\|\e_2\|^2$ and $y\triangleq\|\e_2\|^{1+b} $, so 
\[
  v_2 = {\dm\over 2}y+{\gamma d_{1} \Pm\over 2}x -\gamma_{A1} d_{1}A_0 n^{b-a\over 2(1+b)}y^{\nu_2}x^{1-\nu_2} + \Delta {D}_lw,
\]
with $\nu_1=1/2$, $\nu_2={b\over 1+b}$ and $\Delta {D}_l=\left[\dLm/2-\gamma_{A1} d_{1}\left(\mu_M+\sqrt{n}L_c\right)\right]$. Then, we invoke Lemma \ref{LemaPositiveDefv} below, with $Q_{x}=\dm/2$, $Q_{y}=\Pm/2$, $Q_{w}=\Delta {D}_l$, $q_{w}=\mu_M+\sqrt{n}L_c$, $\varpi_1=A_0n^{b-a\over 2(1+b)}$, $d=d_{1}$, $\varpi_2=0$ and $\varpi_3=0$. 

\begin{lemma}\label{LemaPositiveDefv} For all $x,y,w\in\mathbb{R}_{\geq 0}$, for any constants $\varpi_{1},\varpi_{2},\varpi_{3}\in\mathbb{R}_{\geq0}$ and $Q_{x},Q_{y},Q_{w},q_{w},\alpha,b,d\in \mathbb{R}_{>0}$,  such that $\alpha\in (0,1)$ and $1\geq b>0$,  define the function $v:\mathbb{R}^{3}\to\mathbb{R}$ as follows
\begin{equation}\label{Lemforv}
\begin{array}{rl}
  v(x,y,w):=& Q_{y}y+\gamma dQ_{x}x -\gamma d\varpi_{1} y^{\nu_{2}}x^{1-\nu_{2}}+(Q_{w}-\gamma dq_{w})w-(\varpi_{2} +\gamma d\varpi_{3}){w}^{\nu_{1}} x^{1-\nu_{1}},
\end{array}
\end{equation}
where  $\nu_{1}=1/2$ and $\nu_{2}={b\over 1+b}$. The function $v$ is positive definite provided that the coefficients are chosen as follows: 
\begin{itemize}
	\item If $\varpi_{1},\varpi_{2},\varpi_{3}\in\mathbb{R}_{>0}$, then choose the coefficients such that 
	\begin{subequations}
\begin{align}
	\alpha Q_{y}Q_{w}>&\varpi_{2}\varpi_{3}+\varpi_{2}^{2}{q_{w}\over Q_{w}}, \label{gaincondition1}\\
Q_{y}>&\left[{ \varpi_{1}(1-\nu_{2})\over (1-\alpha)Q_{x}}\right]^{ 1-\nu_{2}\over \nu_{2}}{d\varpi_{1}\nu_{2}}	\left(\frac{B}{2}-\frac{1}{2}\sqrt{B^{2}-4C}\right), 
	\label{gaincondition2}\\
	\min\left\{
\Gamma_{1},\Gamma_{2}\right\}>& \gamma > \frac{B}{2}-\frac{1}{2}\sqrt{B^{2}-4C},\label{Condongama}
\end{align}
\end{subequations}
where
\begin{subequations}\label{Constantofnu}
\begin{align}
\Gamma_{1}\triangleq&\left[{ (1-\alpha)Q_{x}\over \varpi_{1}(1-\nu_{2})}\right]^{ 1-\nu_{2}\over \nu_{2}}{Q_{y}\over d\varpi_{1}\nu_{2}}, \qquad 
\Gamma_{2}(\varpi_{2},\varpi_{3})\triangleq\frac{B}{2}+\frac{1}{2}\sqrt{B^{2}-4C}	
\end{align}
\end{subequations}
and
\begin{equation}\label{definitionbc}
B={\alpha Q_{x}Q_{w}-\frac{1}{2}\varpi_{2}\varpi_{3}\over
d\left[\frac{1}{4}\varpi_{3}^{2}+\alpha Q_{x}q_{w}\right]}, \qquad 
C={\frac{1}{4}\varpi_{2}^{2}\over
d^{2}\left[\frac{1}{4}\varpi_{3}^{2}+\alpha Q_{x}q_{w}\right]}.	
\end{equation}
\item If $\varpi_{1},\varpi_{3}\in\mathbb{R}_{>0}$ and $\varpi_{2}=0$,  then choose $\gamma$ such that
\begin{equation}\label{Gamacase2}
	\min\left\{
\Gamma_{1},\Gamma_{2}(0,\varpi_{3})\right\}>\gamma>0.
\end{equation}
\item If $\varpi_{1}>0$ and $\varpi_{2}=\varpi_{3}=0$,  then  choose $\gamma$ such that
\begin{equation}\label{Gamacase3}
	\min\left\{\Gamma_{1},
	\Gamma_{2}(0,0)\right\}>\gamma>0.
\end{equation}
\end{itemize}
Moreover,  $v(x,y,w)\geq v_{l}(x,y,w)$ where the positive definite function $v_{l}(x,y,w)$ is defined as
\begin{equation}\label{Lemforvlower}
\begin{array}{c}
  v_{l}(x,y,w):= A_{y}y+\gamma dB_{x}x+B_{\bar x}{x}+A_{w}w,
\end{array}
\end{equation}
with
\begin{subequations}\label{Constantofvl}
\begin{align}
A_{y}=&Q_{y}-\gamma d\varpi_{1}\nu_{2}c_{2}^{-{1\over \nu_{2}}},\\
 B_{x}=&(1-\alpha) Q_{x}
 -\varpi_{1}(1-\nu_{2})c_{2}^{1\over 1-\nu_{2}},\\
 B_{\bar x}=&\alpha\gamma d Q_{x}
 -(\varpi_{2} +\gamma d\varpi_{3})(1-\nu_{1})c_{1}^{{1\over 1-\nu_{1}}},\\
 A_{w}=&Q_{w}-\gamma d q_{w}-(\varpi_{2} +\gamma d\varpi_{3})\nu_{1}c_{1}^{-{1\over \nu_{1}}}
\end{align}	
\end{subequations}
and the constants $c_{1}$ and $c_{2}$ satisfy
\begin{subequations}\label{Condonc1yc2}
\begin{align}
\left[ {(1-\alpha)Q_{x}\over\varpi_{1}(1-\nu_{2})}\right]^{1-\nu_{2}\over \nu_{2}}>&c_{2}^{1\over \nu_{2}}>{\gamma d\varpi_{1}\nu_{2}\over Q_{y}},\label{condonc2}\\
{\alpha \gamma d Q_{x}\over \left[\gamma d\varpi_{3}+\varpi_{2}\right](1-\nu_{1})}>&c_{1}^{1\over1-\nu_{1}}	>\left[{\left[\gamma d_{1}\varpi_{3}+\varpi_{2}\right]\nu_{1}\over Q_{w}-\gamma d q_{w}}\right]^{\nu_{1}\over 1-\nu_{1}}\label{condonc1}.	
\end{align}
\end{subequations}
\end{lemma}
\begin{proof} Applying Lemma \ref{LemYoungMor} to the crossed terms, we have that
\[
\begin{array}{rl}
    w^{\nu_{1}}{x}^{1-\nu_{1}}\leq&  \nu_{1}c_{1}^{-{1\over \nu_{1}}}w +(1-\nu_{1})c_{1}^{1\over 1-
    \nu_{1}}{x},\\
   y^{\nu_{2}}x^{1-\nu_{2}}\leq&  \nu_{2}c_{2}^{-{1\over \nu_{2}}} y+(1-\nu_{2})c_{2}^{1\over 1-\nu_{2}}x.
\end{array}
\]
From these inequalities and any real number $\alpha \in (0,1)$ we have that $v(x,y,w)$ is lower bounded by (\ref{Lemforvlower}) and $v(x,y,w)\succ0$ if $A_{y}y+\gamma d_{1}B_{x}y+B_{\bar x}{x}+A_{w}w\succ 0$, where the coefficients $A_{y},B_{x},B_{\bar x}$
 and $A_{w}$ are given in (\ref{Constantofvl}).
To guarantee $v(x,y,w)\succ 0$, the constants $A_{y}$, $A_{w}$, $B_{\bar x}$ and $B_{x}$ have to be positive and they imply 
\[
\begin{array}{lcl}
A_{y}>0\iff c_{2}^{1\over \nu_{2}}>{\gamma d\varpi_{2} \nu_{2}\over Q_{y}},& & B_{x}>0\iff {(1-\alpha) Q_{x}\over \varpi_{1}(1-\nu_{2})}>c_{2}^{1\over 1-\nu_{2}},\\
 A_{w}>0\iff  c_{1}^{1\over \nu_{1}}>{(\varpi_{2} +\gamma d\varpi_{3})\nu_{1}\over Q_{w}-\gamma d q_{w}}	& & B_{\bar x}>0\iff   {\alpha \gamma d Q_{x}\over (\varpi_{2} +\gamma d\varpi_{3})(1-\nu_{1})}>c_{1}^{1\over 1-\nu_{1}}.
\end{array}
\]
These inequalities give rise to (\ref{Condonc1yc2}). On one hand, from (\ref{condonc2}) and by the transitivity property of inequalities, there always exists $c_{2}$ such that $A_{x}>0$ and $B_{y}>0$ if
\[
\left[ {(1-\alpha)Q_{x}\over\varpi_{1}(1-\nu_{2})}\right]^{1-\nu_{2}\over \nu_{2}}>{\gamma d\varpi_{1}\nu_{2}\over Q_{y}},
\]
or equivalently, if 
\begin{equation}\label{CondonGainD}
\Gamma_{1}>\gamma .	
\end{equation}

On the other hand, from (\ref{condonc1}), there always exists $c_{1}$ such that $A_{w}>0$ and $B_{\bar y}>0$ if
\[{\alpha \gamma d Q_{x}\over \left[\gamma d \varpi_{3}+\varpi_{2}\right](1-\nu_{1})}>	\left[{(\gamma d \varpi_{3}+\varpi_{2})\nu_{1}\over Q_{w}-\gamma d q_{w}}\right]^{\nu_{1}\over 1-\nu_{1}}.
\]
Since $\nu_{1}=1/2$, it is equivalent to $0>C-B\gamma+\gamma^{2}$ where $C$ and $B$ are given in (\ref{definitionbc}). Two cases are now taken into account:

\begin{itemize}
	\item {Case $B>0$ and $C\neq 0$:} Real positive solutions for $\gamma$ exist if $\alpha Q_{x}Q_{w}>\frac{1}{2}\varpi_{2}\varpi_{3}$, and if $B^{2}-4C>0$, i.e.
\[
\begin{array}{cc}
  \left[\alpha Q_{x}Q_{w}-\frac{1}{2}\varpi_{2}\varpi_{3}\right]^{2}>\varpi_{2}^{2}\left[\frac{1}{4}\varpi_{3}^{2}+\alpha Q_{x}q_{w}\right],
\end{array}
\]
which implies condition (\ref{gaincondition1}). We need to assure that conditions (\ref{CondonGainD}) and 
\begin{equation}\label{IntervalOfgam}
\gamma\in\left(\frac{B}{2}-\frac{1}{2}\sqrt{B^{2}-4C},\Gamma_{2}\right)
 \end{equation}
hold simultaneously for the existence of $\gamma$. This implies that 
\[
\Gamma_{1}>\gamma>\frac{B}{2}-\frac{1}{2}\sqrt{B^{2}-4C},
\]
or equivalently, we have to choose $\gamma$ as in (\ref{Condongama}), which can be always fulfilled because the r.h.s. of (\ref{gaincondition2})
does not depend on $Q_{y}$. Once the coefficients have been selected to satisfy condition (\ref{gaincondition1}) and (\ref{gaincondition2}), the parameter $\gamma$ must be chosen such that (\ref{Condongama}) holds.

\item {Case $B>0$ and $C= 0$:} This happens when $\varpi_{2}=0$, it is evident that the parameter $\gamma$ must be chosen such that (\ref{Gamacase2}) holds.
\end{itemize}

Finally, in the case $\varpi_{2}=\varpi_{3}=0$, we obtain condition (\ref{Gamacase3}). 
\end{proof}

\section{Technical lemmata}
\label{app:lemmata}

\begin{lemma}
\label{LemYoungMor}(Young's Inequality)\cite{Moreno2011}.
For each $x,y\in \mathbb R_{\geq 0}$, $c\in \mathbb R_{\geq 0}$ and any real number $\mu\in (0,1)$, the inequality  $x^{\mu}y^{1-\mu}\leq \mu c^{1\over\mu}x+(1-\mu)c^{-{1\over 1-\mu}}y$ holds.
\end{lemma} 

\begin{lemma}\label{LemaPosV}
For all $x,y\in \mathbb{R}$ and  each $A,B,C,d_0,r,s,m\in \mathbb{R}_{>0}$ such that $m>s$, define the functions $v_{i}:\mathbb{R}^2\to\mathbb{R}$, $i=\{1,2\}$  by $
v_{1}(x,y) = v_0^{m/2s}(x,y) +
d_0 C \lceil x \rfloor^{(m-s)/r}y$ and $v_{2}(x,y) = v_0^{m/2s}(x,y) -
d_0 C |x|^{\frac{m-s}{r}}|y|$, where $v_0(x,y)\triangleq A|y|^2+B|x|^{\frac{2s}{r}}$. If $d_0$ satisfies $d_0^2 \leq (A m/C^2s)$ $
\left[Bm/(m-s)\right]^{(m-s)/s}$,  then both $v_{1}$ and $v_{2}$ are non-negative. 
Moreover, $v_{1}(x,y)\leq 2v_0^{m/2s}(x,y)$ for all $x,y\in\mathbb{R}$.  
\end{lemma}

\begin{lemma}[See, e.g., \cite{CruNuMo2021}]\label{lemaMeanP}
For any ${\bf z}\in \mathbb{R}^{n}$, the following inequalities hold: (i) $\|{\bf z}\|^{2\alpha}\leq  \sum_{i=1}^{n}\left(|z_{i}|^2\right)^{\alpha}\leq  n^{1-\alpha}\|{\bf z}\|^{2\alpha}, \forall\alpha\in(0,1]$; and (ii) $\|{\bf z}\|^{1+\alpha}\leq  \sum_{i=1}^{n}\left(|z_{i}|^2\right)^{1+\alpha\over 2}\leq  n^{1-\alpha\over 2}\|{\bf z}\|^{1+\alpha}, \forall\alpha\in[0,1]$. 
\end{lemma}

\begin{lemma}[See, e.g., \cite{Bernstein2009}]\label{lemaMeanP0}
For any $x_{1},...,x_{n}\in \mathbb{R}_{\geq0}$, the following inequalities hold: (i) $\left(\sum_{i=1}^{n}x_{i}\right)^{a}\leq \sum_{i=1}^{n}x_{i}^{a}\leq n^{1-a}\left(\sum_{i=1}^{n}x_{i}\right)^{a}, \forall a\in(0,1]$; and (ii) $n^{1-a}\left(\sum_{i=1}^{n}x_{i}\right)^{a}\leq \sum_{i=1}^{n}x_{i}^{a}\leq \left(\sum_{i=1}^{n}x_{i}\right)^{a}, \forall a\geq1$.
\end{lemma}

\begin{lemma}\label{lemnormprodvectors}
For any $a,b\in\mathbb{R}_{>0}$  and a vector $|\y|^{b}\lceil \x\rfloor^{a}:=\begin{bmatrix}
|y_{1}|^{b}\lceil x_{1}\rfloor^{a} &	  ,\dots &, |y_{n}|^{b}\lceil x_{n}\rfloor^{a} 
\end{bmatrix}^{\top}\in\mathbb{R}^{n}$, the following inequality holds  $\||\y|^{b}\lceil \x\rfloor^{a}\|\leq c_{1}c_{2}\|\y\|^{b}\|\x\|^{a}$, where $c_{1}=n^{1-2a\over 4}$  and  $c_{2}=n^{1-2b\over 4}$ for all $a\in(0,1/2]$, $b\in(0,1/2]$; and $c_{1}=1$ and $c_{2}=1$ for all $a\geq1/2$ and $b\geq1/2$.
\end{lemma}
\begin{proof} 
Note that $\||\y|^{b}\lceil \x\rfloor^{a}\|:=\left(\sum_{i=1}^{n} |y_{i}|^{2b}|x_{i}|^{2a} \right)^{1/2}=\left([|\y|^{2b}]^{\top}|\x|^{2a} \right)^{1/2}$. From the Cauchy-Schwarz's inequality we have that $[|\y|^{2b}]^{\top}|\x|^{2a}\leq \||\y|^{2b}\|\||\x|^{2a}\|$. From Lemma \ref{lemaMeanP0}, it follows that $\||\x|^{2a}\|\leq n^{1-2a\over 2}\|\x\|^{2a}$ for all $a\in(0,1/2]$ and $\||\x|^{2a}\|\leq\|\x\|^{2a}$ for all $a\geq1/2$, and similarly for $\||\y|^{2b}\|$. Then, we have that $\||\y|^{b}\lceil \x\rfloor^{a}\|\leq c_{1}c_{2}\left(\|\y\|^{2b}\|\x\|^{2a} \right)^{1/2}=c_{1}c_{2}\|\y\|^{b}\|\x\|^{a}$, where  $c_{1}=n^{1-2a\over 4}$ for all $a\in(0,1/2]$ and $c_{1}=1$ for all $a\geq1/2$; and $c_{2}=n^{1-2b\over 4}$ for all $b\in(0,1/2]$ and $c_{2}=1$ for all $b\geq1/2$.	
\end{proof}

The following lemmas establish properties of homogeneous functions.

\begin{lemma}[\cite{Hes66}]\label{lemAndrieuH}
Let $\eta:\mathbb{R}^{m}\rightarrow\mathbb{R}$ and $\vartheta:\mathbb{R}^{m}\rightarrow\mathbb{R}_{\geq0}$ be two $(\mathbf{r},l)$-homogeneous continuous functions. Define $Z=\{{\bf x}\in\mathbb{R}^{n}\setminus\{\bf 0\}\, : \, \vartheta({\bf x})=0\}$. If $\eta({\bf x})>0$ for all ${\bf x}\in Z$, then there exist $\lambda^{\ast},c\in \mathbb{R}_{>0}$
such that for all $\lambda\geq \lambda^{\ast}$ and for all ${\bf x}\in\mathbb{R}^{m}\setminus\{\bf 0\}$, $\eta ({\bf x})+\lambda\vartheta({\bf x})> c \|{\bf x}\|^{l}_{{\bf r},p}$, where for any $p\ge1$ and each ${\bf x}\in\mathbb{R}^m$, $\|{\bf x}\|_{{\bf r},p}:=\left(\sum\limits_{i=1}^m |x_i|^{p/r_i}\right)^{1/p}$ is the ${\bf r}$-{\em homogeneous norm}. 
\end{lemma}

\begin{lemma}[\cite{Bacciotti2005}]\label{LemmaBhatV1V2}
Let $V_k$, $k=1,2$ be two  $(\textbf{r},l_{k})$-homogeneous continuous functions, with $l_k\in \mathbb{R}_{>0}$.\\
(a) If $V_{1}$ is positive definite, then there are constants $c_{1},c_{2}\in\mathbb{R}$ such that
$c_{1}[V_{1}\left(\mathbf{x}\right)]^{l_{2}/l_{1}}\leq V_{2}\left(\mathbf{x}\right)\leq c_{2}[V_{1}\left(\mathbf{x}\right)]^{l_{2}/l_{1}}$ for all $\mathbf{x}\in\mathbb{R}^m$. Moreover, if $V_{2}$ is positive definite, then $c_{1},c_2\in \mathbb{R}_{>0}$. \\
(b) If $V_1\in {\mathcal C}^k$, then the partial derivatives $\partial_{x_i} V(\mathbf{x})$ are $({\bf{r}}, l_V-r_i)$-homogeneous. 

\end{lemma}

The following three lemmas establish properties of \emph{bl}-homogeneous functions.

\begin{lemma} \label{corAndrieuP2008}\cite{AndrieuP2008}. Let $\phi : \mathbb{R}^{m}\rightarrow \mathbb{R}$ and $\zeta : \mathbb{R}^{m}\rightarrow \mathbb{R_{\geq 0}}$ be two homogeneous in the bi-limit functions with the same weights $r_{0}$ and $r_{\infty}$, degrees $d_{\phi,0}$, $d_{\phi,\infty}$ and $d_{\zeta,0}$, $d_{\zeta,\infty}$
and approximating functions $\phi_{0}$, $\phi_{\infty}$ and $\zeta_{0}$, $\zeta_{\infty}$. If the degrees satisfy
$d_{\phi,0}\geq d_{\zeta,0}$ and $d_{\phi,\infty}\leq d_{\zeta,\infty}$ and the functions $\zeta$, $\zeta_{0}$ and $\zeta_{\infty}$ are positive definite, then there exists a
positive real number $\lambda$ satisfying $
  \phi(\x)\leq\lambda\zeta(\x), \ \forall \x \in \mathbb{R}^{m}$.
\end{lemma}
This result establishes that \emph{bl}-homogeneous functions can be 
bounded by positive-definite \emph{bl}-homogeneous functions. In fact, if $\phi(\x)$ is positive definite then there exist $ \bar{c}_{3}>0$ such that  $
  \bar{c}_{3}\phi(\x)\leq\zeta(\x), \ \forall \x \in \mathbb{R}^{m}$.

The following result is a particular case of Lemma \ref{corAndrieuP2008}. 

\begin{lemma}\label{LemaBilimit}
Let $x\in \mathbb{R}_{\geq0}$ and let $a,b,c,\lambda_{1},\lambda_{2},\lambda_{3}\in \mathbb{R}_{>0}$ such that $c>b>a$. If $
\lambda_{1}\leq{\overline\lambda}_{1}(\lambda_{2},\lambda_{3},a,b,c)\triangleq\left({ (c-a)\lambda_{3}\over(b-a)\lambda_{2}}\right)^{{b-a\over c-a}}{(c-a)\lambda_{2}\over c-b}$, 
then the inequality $\lambda_{1}x^{b}\leq \lambda_{2}x^{a}+\lambda_{3}x^{c}$ holds.
\end{lemma}
\begin{proof} Write the inequality as $f(x)\leq 1/\lambda_{1}$, with $f(x)={x^{b-a}\over \lambda_{2}+\lambda_{3}x^{c-a}}$. By simple calculus, we find that $\max\{f(x)\}=\left({(b-a)\lambda_{2}\over (c-a)\lambda_{3}}\right)^{{b-a\over c-a}}{ c-b\over (c-a)\lambda_{2}}$.
\end{proof}

\begin{lemma}\label{AndrieuP2008CompF}[Composition function, \cite{AndrieuP2008}] .
If $\varphi : \mathbb{R}^{m} \to \mathbb{R}$ and $\zeta : \mathbb{R} \to \mathbb{R}$ are
homogeneous in the $0$-limit (respectively, $\infty$-limit) functions, with weights
$r_{\varphi,0}$ and $r_{\zeta,0}$ (respectively, $r_{\varphi,\infty}$ and $r_{\zeta,\infty}$),
degrees $d_{\varphi,0} > 0$ and $d_{\zeta,0} \ge 0$
(respectively, $d_{\varphi,\infty} > 0$ and $d_{\zeta,\infty} \ge 0$),
and approximating functions $\varphi_0$ and $\zeta_0$
(respectively, $\varphi_\infty$ and $\zeta_\infty$),
then $\zeta \circ \varphi$ is homogeneous in the $0$-limit  (respectively, $\infty$-limit) 
with weight $r_{\varphi,0}$ (respectively, $r_{\varphi,\infty}$),
degree
$d_{\zeta,0} d_{\varphi,0}/r_{\zeta,0}$ (respectively,  $d_{\zeta,\infty} d_{\varphi,\infty}/r_{\zeta,\infty}$)
and approximating function $\zeta_0 \circ \varphi_0$
(respectively, $\zeta_\infty \circ \varphi_\infty$).	
\end{lemma}

\begin{lemma} [See, e.g., \cite{CuzMoToA2025}] \label{LemaVnn2}
For all $ \x,\y,\z \in\mathbb{R}^{n}$ and for any constants $r,s\in\mathbb{R}_{>0}$ such that $2s \geq r \geq s$, define the function  $\bar{{\mathcal{H}}}_d:{\mathbb{R}}^n \times {\mathbb{R}}^n \times {\mathbb{R}}^n\rightarrow {\mathbb{R}_{\geq 0}}$ as follows
\[ {\bar{\mathcal H}}_{d}(\x,\y,\z) := \frac{1}{2}\y^{\top}\M(\z)\y+{r\over 2s}\x^{\top}\P \lceil \x\rfloor ^{a_{sr}\over r}, \qquad a_{sr}= 2s-r,
\] 
where $\M(\z)\in \mathbb{R}^{n\times n}$ is a matrix such that, for all $\z\in\mathbb{R}^n$, there exist constants $M_{1},M_{2}\in\mathbb{R}_{>0}$ satisfying $M_{1}\I\leq\M(\z)\leq M_{2}\I$, and the matrix $\P$ is constant, diagonal and positive definite. Then:
\begin{itemize}
\item  Function ${\bar{\mathcal H}}_{d}$ satisfies $\mathcal H_{l}(\x,\y) \leq \bar{\mathcal H}_{d}(\x,\y,\z)\leq \mathcal H_{u}(\x,\y)$, where
\begin{subequations}\label{BoundsonVContCase}
	\begin{align}
	   \mathcal H_{l}(\x,\y) := &\tfrac{M_{1}}{2}\|\y\|^{2}+\tfrac{r\lambda_{m}\{\P\}}{ 2s}\|\x\|^{2s\over r},\\
   \mathcal H_{u}(\x,\y) := &\tfrac{M_{2}}{2}\|\y\|^{2}+{r{n^{r-s\over 2r}}\lambda_{M}\{\P\}\over 2s}\|\x\|^{2s\over r}.	\end{align}
\end{subequations}
 
\item For any given constants $m,d_{0}\in \mathbb{R}_{>0}$ such that $ m\geq s+r$  and
\begin{equation}\label{d1value0}
d_{0}^{2}\leq {M_{1}m\over 2M_{2}^{2}s}\left[{m\over m-s}\tfrac{r\lambda_{m}\{\P\}}{ 2s}\right]^{m-s\over s},
\end{equation}
the following inequality holds for all $ \x,\y,\z \in\mathbb{R}^{n}$:
\begin{equation}\label{IneqLem}
2{\bar{\mathcal H}}^{m\over 2s}_d(\x,\y,\z) \geq 
{\bar{\mathcal H}}^{m\over 2s}_d(\x,\y,\z) +d_{0} \lceil \x^{\top}\rfloor^{m-s\over r} \M(\z)\y \geq 0. 
\end{equation}
\end{itemize}
\end{lemma}

\begin{lemma}\label{Lemboundtnh} For any $y\in{\mathbb R}_{>0}$, the following inequality holds: $y/(1+y)\leq \tanh(y)\leq K_{t}y/(1+y)$, where $K_{t}>c_{M}\approx 1.52855$.
\end{lemma}
\begin{proof} Define $f(y)=\tanh(y)(1+y)/y$. By simple calculus it has  a minimum at $y=0$ and a maximum $c_{M}=f(y)$ at $y=y_{max}\approx 1.1523$.	
\end{proof}
\bibliographystyle{plain}
\bibliography{CMLM_IJRNC-2026}

\end{document}